\documentclass[12pt]{amsart}  	

\usepackage[all]{xy}						
\usepackage{amssymb}
\usepackage{amscd,latexsym,amsthm,amsfonts,amssymb,amsmath,amsxtra}
\usepackage[mathscr]{eucal}

\usepackage[colorlinks]{hyperref}
\hypersetup{linkcolor=black, citecolor=black}

\newcommand{\BC}{{\mathbb {C}}}

\newcommand{\BG}{{\mathbb {G}}}

\newcommand{\BR}{{\mathbb {R}}}

\newcommand{\BZ}{{\mathbb {Z}}}

\newcommand{\CC}{{\mathcal {C}}}

\newcommand{\CF}{{\mathcal {F}}}

\newcommand{\CH}{{\mathcal {H}}}

\newcommand{\CO}{{\mathcal {O}}}
\newcommand{\CP}{{\mathcal {P}}}

\newcommand{\CS}{{\mathcal {S}}}

\newcommand{\Fa}{{\mathfrak {a}}}

\newcommand{\Fg}{{\mathfrak {g}}}

\newcommand{\Fk}{{\mathfrak {k}}}

\newcommand{\Fp}{{\mathfrak {p}}}

\newcommand{\RM}{{\mathrm {M}}}

\newcommand{\GL}{{\mathrm{GL}}}

\renewcommand{\Im}{{\mathrm{Im}}}

\newcommand{\Id}{{\mathrm{Id}}}

\newcommand{\PGL}{{\mathrm{PGL}}}

\renewcommand{\Re}{{\mathrm{Re}}}

\newcommand{\Sym}{{\mathrm{Sym}}}

\newcommand{\tr}{{\mathrm{tr}}}

\newcommand{\wt}{\widetilde}
\newcommand{\wh}{\widehat}

\newcommand{\bs}{\backslash}

\def\del{{\delta}}

\def\eps{{\epsilon}}

\def\veps{{\varepsilon}}

\def\sig{{\sigma}}

\def\zet{{\zeta}}
\def\std{\rm std}

\def\lam{{\lambda}}

\def\gam{{\gamma}}
\def\Gam{{\Gamma}}

\def\wb{\overline} 

\def\vpi{\varpi}
\def\LG{{}^{L}G}
\def\LT{{}^{L}T}

\def\vphi{\varphi}

\def\p{\prime}

\newtheorem{thm}{Theorem}[subsection]
\newtheorem{defin}[thm]{Definition}
\newtheorem{rmk}[thm]{Remark}

\newtheorem{pro}[thm]{Proposition}
\newtheorem{lem}[thm]{Lemma}

\newtheorem{conjec}[thm]{Conjecture}

\makeatletter

\newcommand{\Rmnum}[1]{\expandafter\@slowromancap\romannumeral #1@}
\makeatother

\begin{document}

\title[BK-Spherical]{On the Braverman-Kazhdan Proposal for Local Factors: Spherical Case}
\author{Zhilin Luo}
\address{School of Mathematics\\
University of Minnesota\\
Minneapolis, MN 55455, USA}
\email{luoxx537@umn.edu}

\begin{abstract}
In this paper, we study the Braverman-Kazhdan proposal for the local spherical situation. In the $p$-adic case, we give a definition of the spherical component of conjectural space $\CS_{\rho}(G,K)$ and the $\rho$-Fourier transform kernel $\Phi_{\rho}^{K}$, and verify several conjectures in \cite{BK00} in this situation. In the archimedean case, we study the asymptotic of the basic function $1_{\rho,s}$ and the $\rho$-Fourier transform kernel $\Phi_{\rho,s}^{K}$.
\end{abstract}

\maketitle

\tableofcontents

\section{Introduction}
The theory of zeta integrals can be traced back to the work of B. Riemann, who first wrote the Riemann zeta function $\zet(s) = \sum_{n=1}^{\infty}\frac{1}{n^{s}}$ as the Mellin transform of a theta function. The idea was developed by J. Tate in his thesis \cite{tatethesis} using the theory of zeta integrals. For convenience, in the introduction we restrict to the non-archimedean local fields case. For each character $\chi$ of $F^{\times}$, where $F$ is a non-archimedean local field, one considers a family of distributions given by zeta integrals $Z(s,f,\vphi)$ with parameter $s\in \BC$ on the space $\CS(F^{\times}) = C^{\infty}_{c}(F)$. Tate shows that the distribution admits meromorphic continuation to $s\in \BC$, possibly with a pole at $s=0$. The pole can be described by the $L$-factor $L(s,\chi)$, in the sense that the distribution
$\frac{Z(s,\cdot,\chi)}{L(s,\chi)}$
admits holomorphic continuation to the whole complex plane.

R. Godement and H. Jacquet \cite{gjzeta} generalize the work of Tate and study, for any irreducible admissible representation $\pi$ of $\GL(n)$ over a non-archimedean local field, the family of distributions given by zeta integrals $Z(s,f,\vphi_{\pi})$ with parameter $s\in \BC$ on the space $\CS(\GL(n)) = C^{\infty}_{c}(\RM_{n})$, where $\vphi_{\pi}\in \CC(\pi)$ is a matrix coefficient of $\pi$. They show that $Z(s,f,\vphi_{\pi})$ has meromorphic continuation to $s\in \BC$ with a possible pole at $s=0$, and their poles are captured by the standard local $L$-factor $L(s,\pi)$ attached to $\pi$.

According to R. Langlands (\cite{langlandsproblems}), for any reductive algebraic group $G$ defined over $F$, and for any finite dimensional representation $\rho$ of the Langlands dual group $\LG$, one may define the local $L$-factor $L(s,\pi,\rho)$ associated to an irreducible admissible representation $\pi$ of $G(F)$. It is natural to ask: Is it possible to find a family of distributions similar to the case of Godement-Jacquet that define the general local $L$-factor $L(s,\pi,\rho)$? Over the last fifty years, one found various types of global zeta integrals of Rankin-Selberg type, whose local zeta integrals may define local $L$-factors for a special list of $G$ and $\rho$. Often, the zeta integrals of Rankin-Selberg type are not the same as that of Godement-Jacquet. In 2000, A. Braverman and D. Kazhdan in \cite{BK00} propose a conjectural construction of families of distributions that may define the general $L$-factors $L(s,\pi,\rho)$, similar to that in \cite{gjzeta}. We will explain their proposal below.

\subsection{Notation and Convention}

Throughout the paper, we fix a local field $F$ of characteristic $0$, which can be either a $p$-adic field or an archimedean field. When $F$ is a $p$-adic field, we let $\CO_{F}$ be the ring of integers of $F$ with fixed uniformizer $\vpi$, and we assume that the residue field of $F$ has cardinality $q$.

We fix a valuation $|\cdot|$ on $F$. When $F$ is a $p$-adic field, we normalize $|\cdot|$ so that $|\vpi| = q^{-1}$. When $F\cong \BR$, it is the usual valuation on $\BR$. When $F\cong \BC$, $|z| = z\wb{z}$ for any $z\in \BC$, where $\wb{z}$ is the complex conjugate of $z$.

Let $G$ be a split connected reductive algebraic group over $F$. Following the notation of \cite[Section~3.1]{wwlibasic}, we assume that the group $G$ fits into the following short exact sequence
\begin{align}\label{shortexact}
\xymatrix{
1 \ar[r] & G_{0} \ar[r] & G \ar[r]^{\sig} & \BG_{m} \ar[r] &1
}
\end{align}
Here $G_{0}$ is a split connected semisimple algebraic group over $F$, and $\sig$ is a character of $G$ playing the role of determinant as in $\GL(n)$ case.

Let $\LG$ be the Langlands dual group of $G$. We fix an irreducible algebraic representation
\begin{align*}
\rho: \LG\to \GL(V_{\rho})
\end{align*}
of dimension $n=\dim V_{\rho}$. There are similar results for reducible $\rho$, but for convenience we only work with the case when $\rho$ is irreducible.
Following \cite[Definition~3.13]{BK00} and \cite[Section~3.1]{wwlibasic}, we further assume that $\rho$ is faithful, the restriction of $\rho$ to the central torus $\BG_{m}\to \LG$ is $z\to z\Id$, and $\ker (\rho)$ is connected.

We require that the representation $\rho$ fits into the following commutative diagram
\begin{align*}
\xymatrix{
1 \ar[r]  & \BG_{m} \ar[d]^{\Id} \ar[r]^{\wh{\sig}}&  \LG  \ar[d]^{\rho} \ar[r]
& \LG_{0} \ar[d]^{\wb{\rho}} \ar[r] & 1\\
1 \ar[r] & \BG_{m} \ar[r] & \GL(V,\BC) \ar[r] & \PGL(V,\BC) \ar[r] & 1
}
\end{align*}
The top row is obtained by dualizing the short exact sequence (\ref{shortexact}),
and $\wb{\rho}$ is the projective representation obtained from $\rho$. By \cite[Section~3.1]{wwlibasic}, we may assume that the rows are exact and the second square is cartesian.

We fix a Borel pair $(B,T)$ for our group $G$. Let $X_{*}(T)$ and $X^{*}(T)$ be the cocharacter and character group of $T$ respectively. Let $W=W(G,T)$ be the Weyl group. Let $\rho_{B}$ be the half sum of positive roots. The corresponding modular character is denoted by $\del_{B}$. Following the suggestion of \cite{bouthier2016formal} and \cite{formalarcerrtum}, we let $l = 2<\rho_{B},\lam>$, where $\lam$ is the highest weight of the representation $\rho$.

When $F$ is a $p$-adic field, we choose a hyperspecial vertex in the Bruhat-Tits building of $G$ which lies in the apartment determined by $T$. The corresponding hyperspecial subgroup $G(F)$ is denoted by $K$
 as usual. When $F$ is an archimedean field, by Cartan-Iwasawa-Malcev theorem \cite[Theorem~1.2]{borelssgrp}, we fix a maximal compact subgroup $K$ of $G$.

When $F$ is a $p$-adic field, we fix the Cartan decomposition $G(F)=\coprod_{\lam\in X_{*}(T)_{+}} K\lam(\vpi)K$, where $X_{*}(T)_{+}$ is the positive Weyl chamber. When $F$ is an archimedean field, we also fix the Cartan decomposition $G=K\exp(\Fa)K$, where $\Fa$ is a maximal abelian subaglebra of the Lie algebra $\Fg$ of $G$.
Let $T(F)\cap K = T_{K}$.

We fix a nontrivial additive character $\psi$ of $F$ with conductor $\CO_{F}$. We also fix a Haar measure on $F$ such that the Haar measure is self-dual w.r.t. the additive character $\psi$.

\subsection{Braverman-Kazhdan Proposal}
In \cite{BK00}, the local aspect of the Braverman-Kazhdan proposal is to construct a family of zeta distributions associated to each finite dimensional representation $\rho$ of the Langlands dual group $\LG$ that define the general $L$-factor $L(s,\pi,\rho)$ for every irreducible admissible representation $\pi$ of $G(F)$ via a generalization of the work of Godement and Jacquet \cite{gjzeta}. Roughly speaking, they proposed the existence of a function space $\CS_{\rho}(G)\subset C^{\infty}(G)$, which should be the space of test functions for the zeta distributions, such that the following conjecture holds
\begin{conjec}\cite[Conjecture~5.11]{BK00}\label{zetaintegral}
With the notation above, the following hold.
\begin{enumerate}
\item For every $f\in \CS_{\rho}(G)$ and every $\vphi\in \CC(\pi)$ the integral
$$
Z(s,f,\vphi) = \int_{G}f(g)\vphi(g)|\sig(g)|^{s+\frac{l}{2}}dg
$$
is absolutely convergent for $\Re(s)\gg 0$.

\item $Z(s,f,\vphi)$ has a meromorphic continuation to $\BC$ and defines a rational function of $q^{s}$.

\item $I_{\pi} = \{ Z(s,f,\vphi) | \quad f\in \CS_{\rho}(G), \vphi\in \CC(\pi) \}$ is a finitely generated non-zero fractional ideal of the ring $\BC[q^{s},q^{-s}]$, where $\CC(\pi)$ is the space of matrix coefficients of $\pi$.
\end{enumerate}
\end{conjec}

\begin{rmk}
In \cite{BK00}, Braverman and Kazhdan defined the number $l$ to be the semisimple rank of $G$. Following the work of \cite{bouthier2016formal} and \cite{formalarcerrtum}, it is suggested that the correct normalization should be $l=2<\rho_{B},\lam>$, where $\lam$ is the highest weight of $\rho$. In the case where $\rho$ is the standard representation of $\GL(n)$, the number $l=n-1$. The definition coincides with the work of Godement and Jacquet \cite{gjzeta}.
\end{rmk}

Assuming that the Conjecture \ref{zetaintegral} holds, one may define the local $L$-factor $L(s,\pi,\rho)$ to be the unique generator of the fractional ideal $I_{\pi}$ of the form $P(q^{-s})^{-1}$, where $P$ is a polynomial such that $P(0)=1$.
Moreover, they also proposed the existence of a Fourier-type transform $\CF_{\rho}$ \cite[Section~5.3]{BK00} that is defined by
\begin{align*}
\CF_{\rho}(f) = |\sig|^{-l-1}(\Phi_{\psi,\rho}*f^{\vee}), \quad f\in C^{\infty}_{c}(G),
\end{align*}
and satisfies the following
\begin{conjec}\cite[Conjecture~5.9]{BK00}\label{operator}
The $\rho$-Fourier transform
$\CF_{\rho}$ extends to a unitary operator on $L^{2}(G,|\sig|^{l+1}dg)$ and the space $\CS_{\rho}(G)$ is $\CF_{\rho}$-invariant. Here the character $\sig$ is defined in (\ref{shortexact}).
\end{conjec}
Here $\Phi_{\psi,\rho}$ is a $G$-stable $\sig$-compact distribution in the sense of \cite[Definition~3.8]{BK00}. After unramified twist, the action of $\Phi_{\psi,\rho,s}$ on the space of $\pi \in $Irr$(G)$ is given by a rational function in $s$, which is the associated local gamma factor $\gam(-s-\frac{l}{2},\pi^{\vee},\rho,\psi)$.

\begin{rmk}
Here we want to make a remark on the $\gam$-factor. Assuming the local Langlands functoriality for $\rho$, we can set
$$
\gam(s,\pi,\rho,\psi) = \gam(s,\rho(\pi),\psi),
$$
where $\rho(\pi)$ is the functorial lifting of $\pi$ along $\rho$. The $\gam$-factor is a rational function in $s$.
Hence, for special values of $s$, for instance $s=-\frac{l}{2}$, there might exist $\pi\in Irr(G)$ such that the constant $\gam(-\frac{l}{2},\rho(\pi),\psi)$ does not exist for $\pi$. In this case, we can take an unramified twist of $\Phi_{\psi,\rho}$, which we denote as $\Phi_{\psi,\rho,s}$. Then the action of $\Phi_{\psi,\rho,s}$ on the space of $\pi$ is given by the local gamma factor $\gam(-s-\frac{l}{2},\pi^{\vee},\rho,\psi)$.
\end{rmk}

\begin{rmk}
In \cite[Section~1.2]{BK00}, Braverman and Kazhdan define the distribution $\Phi_{\psi,\rho,s}$ with the property that its action on the space of $\pi \in $$\rm{Irr}$$(G)$ is given by the local gamma factor $\gam(s,\pi,\rho,\psi)$ with parameter $s\in \BC$. For normalization purpose, we define our $G$-stable distribution $\Phi_{\psi,\rho,s}$ with action on $\pi$ via the scalar $\gam(-s-\frac{l}{2},\pi^{\vee},\rho,\psi)$. In Lemma \ref{distributiongamma} below, we show how to derive the relation between $\gam$-factor and $\Phi_{\psi,\rho,s}$ formally from the conjectural functional equation
\begin{align*}
Z(1-s,\CF_{\rho}(f),\vphi^{\vee}) = \gam(s,\pi,\rho,\psi) Z(s,f,\vphi), \quad f\in \CS_{\rho}(G), \vphi\in \CC(\pi)
\end{align*}
\end{rmk}

In \cite[Section~7]{BK00}, Braverman and Kazhdan give a conjectural algebro-geometric construction of the distribution $\Phi_{\psi,\rho,s}$. It is not difficult to define the distribution $\Phi_{\psi,\rho\circ i,s}$ on $T$ associated to the representation $\rho\circ i$ of $\LT$
\begin{align*}
\xymatrix{{\LT} \ar[r]^{i} & \LG\ar[r]^{\rho} &\GL(V_{\rho})}.
\end{align*}
Since the distribution $\Phi_{\psi,\rho,s}$ is conjectured to be $G$-stable, using the adjoint quotient map $G^{reg}\to T/W$, one can naturally extend it to a distribution on $G$ once the $W$-equivariance of the distribution $\Phi_{\psi,\rho\circ i,s}$ is established as conjectured in \cite[Conjecture~7.11]{BK00}. Then Braverman and Kazhdan conjectured that the construction gives us the distribution $\Phi_{\psi,\rho,s}$ that we want. There is a parallel conjecture in finite field case, and some recent works (\cite{braverman2003sheaves}, \cite{chen2016non}, and \cite{cheng2017conjecture}) confirm the construction.

For the construction of function space $\CS_{\rho}(G)$, Braverman and Kazhdan \cite[Section~5.5]{BK00} expect to use the Vinberg's monoids \cite{vinberg1995reductive}. For each $\rho$, one can construct a reductive monoid $\wb{G}_{\rho}$ containing $G$ as an open dense subvariety, whose unit is just the group $G$, and there is a $G\times G$ equivariant embedding of $G$ into $\wb{G}_{\rho}$ . Here $\wb{G}_{\rho}$ is expected to play the role of $\RM_{n}$ as in \cite{gjzeta}. But for almost all $\rho$, $\wb{G}_{\rho}$ is a singular variety. Hence one cannot simply use the locally constant compactly supported functions on $\wb{G}_{\rho}$ as our conjectural function space $\CS_{\rho}(G)$. Recently there are some works in the function field case (\cite{bouthier2016formal} and \cite{formalarcerrtum}) explaining the relation between the geometry of $\wb{G}_{\rho}$ and the \emph{basic function} in $\CS_{\rho}(G)$.

Assuming the local Langlands functoriality for $\rho$, L. Lafforgue \cite{lafforgue2014noyaux} proposes the definition of $\CS_{\rho}$ and $\CF_{\rho}$ using Plancherel formula. However, the analytical properties of $\CS_{\rho}$ and $\CF_{\rho}$ may not be easily figured out from such an abstract definition.

By the work of Godement and Jacquet \cite{gjzeta}, when $\rho$ is the standard representation of $\GL(n)$ the above conjectures hold. We can take $\CS_{\rho}(G)$ to be the restriction to $\GL(n)$ of functions in $C^{\infty}_{c}(\RM_{n})$, and $\GL(n)$ embeds into $\RM_{n}$ naturally. Here $\RM_{n}$ is the monoid of $n\times n$ matrices which fits into the construction of Vinberg \cite{vinberg1995reductive}. $\CF_{\rho}$ in this case is the classical Fourier transform on $\RM_{n}$ fixing $C^{\infty}_{c}(\RM_{n})$ defined by
\begin{align*}
\CF(f)(g) &= |\det g|^{-n}(\Phi_{\psi, \std}*f^{\vee})(g)
\\&= \int_{\RM_{n}(F)} f(y)\psi(\tr(yg))dy,     \quad f\in C^{\infty}_{c}(\RM_{n})
\end{align*}
where $\Phi_{\psi,\std}(g) = \psi(\tr(g))|\det(g)|^{n}$.
\newline

\paragraph{ \textbf{Basic Function }}
Although the structure of the space $\CS_{\rho}(G)$ is still unclear, there is a distinguished element in the space $\CS_{\rho}(G)$, called basic function, which we will introduce below.

In \cite{gjzeta}, the authors find that the characteristic function $1_{\RM_{n}(\CO_{F})}$ of $\RM_{n}(\CO_{F})$ satisfies the following two properties:
\begin{enumerate}
\item For any spherical representation $\pi$ of $G=\GL(n)$ with Satake parameter $c\in \wh{T}/W$, let $\vphi_{\pi}$ be the associated zonal spherical function, then
\begin{align*}
Z(s,1_{\RM_{n}(\CO_{F})},\vphi_{\pi})=\int_{G}1_{\RM_{n}(\CO_{F})}(g)\vphi_{\pi}(g)|\det g|^{s+\frac{n-1}{2}}dg
\\= \det(1-(c)q^{-s}|V)^{-1} = L(s,\pi,\std).
\end{align*}

\item $\CF_{\std}(1_{\RM_{n}(\CO_{F})}) = 1_{\RM_{n}(\CO_{F})}$.
\end{enumerate}

Let $\CS: \CH(G,K)\to \BC[\wh{T}/W]$ be the Satake transform, which is an isomorphism of algebras. Using the Cartan decomposition, the zeta integral $Z(s,1_{\RM_{n}(\CO_{F})},\vphi_{\pi})$ is equal to the Satake transform of the function $1_{\RM_{n}(\CO_{F})}|\det|^{s+\frac{n-1}{2}}$ evaluated at the Satake parameter $c\in \wh{T}/W$ of $\pi$.
For general $\rho$, one is naturally led to the following definition of the \emph{basic function} $1_{\rho,s}$ with parameter $s\in \BC$.
\begin{defin}\cite[Defintion~3.2.1]{wwlibasic}
The basic function $1_{\rho,s} = 1_{\rho}|\sig|^{s}$ with parameter $s\in \BC$ is the smooth bi-$K$-invariant function on $G$ such that
$$
\CS(1_{\rho,s})(c) = L(s,\pi,\rho)
$$
for any spherical representation $\pi$ of $G$ with Satake parameter $c$, where $\sig$ is the character defined in (\ref{shortexact}).
\end{defin}

Following the work of Godement-Jacquet \cite{gjzeta}, one hopes that the function $1_{\rho,-\frac{l}{2}} = 1_{\rho}|\sig|^{-\frac{l}{2}}$ lies in the function space $\CS_{\rho}(G)$ and has the following property
\begin{conjec}\label{fixed}
$\CF_{\rho}(1_{\rho,-\frac{l}{2}})=1_{\rho,-\frac{l}{2}}$.
\end{conjec}
It is shown in \cite[Lemma~5.8]{BK00} that Conjecture \ref{fixed} holds assuming the compatibility of parabolic descent and $\rho$-Fourier transform \cite[Conjecture~3.15]{BK00}.

One of the reasons that we care about the function $1_{\rho,s}$ is its role in Langlands' Beyond endoscopy program \cite{langlands2004beyond}. When $\Re(s)$ is sufficiently large, we expect to plug it into the Arthur-Selberg trace formula \cite{finis2011spectral}. On the spectral side, we would get a partial automorphic $L$-function. On the geometric side, the weighted orbital integrals of the basic function can tell us information about the automorphic $L$-function. For details the reader is recommended to read \cite{ngo2016hankel} and the last section of \cite{getz2015nonabelian}.
\newline

\subsection{Our Results }
We obtain results uniformly for both $p$-adic and archimedean local field. For convenience, we treat them separately in the following.

\paragraph{\textbf{$p$-Adic Case:}}
We give a construction of the spherical component of the function space $\CS_{\rho}(G)$ and the distribution kernel of $\rho$-Fourier transform $\Phi_{\psi,\rho}$, which we denote by $\CS_{\rho}(G,K)$ and $\Phi_{\psi,\rho}^{K}$.
Here we need to use the extension of \emph{Satake isomorphism} $\CS: \CH(G,K)\to \BC[\wh{T}/W]$ to \emph{almost compactly supported functions} $\CH_{ac}(G,K)$ in the sense of \cite[Proposition~2.3.2]{wwlibasic}, since the $L$-functions and $\gam$-factors are rational functions rather than polynomial functions on $\wh{T}/W$. The functions in $\CS_{\rho}(G,K)$ are not always compactly supported, but always \emph{almost compactly supported}.

\begin{defin}\label{fundef}
Define the function space $\CS_{\rho}(G,K)$ to be
$$\CS_{\rho}(G,K) = 1_{\rho,-\frac{l}{2}}*\CH(G,K).$$
Define the distribution kernel of $\rho$-Fourier transform $\Phi^{K}_{\psi,\rho,s}$ to be
$$\Phi^{K}_{\psi,\rho,s} = 1_{\rho,1+s+\frac{l}{2}}*\CS^{-1}(\frac{1}{L(-s-\frac{l}{2},\pi,\rho^{\vee})}).$$
\end{defin}
In Proposition \ref{basicingredient}, we show that when $\rho$ is the standard representation of $G=\GL(n)$, we actually have
$$\CS_{\std}(G,K) = 1_{\textrm{\std},-\frac{n-1}{2}}*\CH(G,K) = 1_{\RM_{n}(\CO_{F})}*\CH(G,K).$$
Here $\CS_{\std}(G,K)$ is the restriction of functions in $C^{\infty}_{c}(\RM_{n},K)$, the bi-$K$-invariant functions in $C^{\infty}_{c}(\RM_{n})$, to $\GL(n)$. The structure for the standard case will be our main ingredient for introducing Definition \ref{fundef}.

Based on Definition \ref{fundef} we can verify that the Conjecture \ref{zetaintegral} and  Conjecture \ref{operator} hold under the assumption that the functions and representations are spherical. We can also verify Conjecture \ref{fixed} without referring to \cite[Conjecture~3.15]{BK00}. More precisely, the following theorems holds
\begin{thm}\label{gcd}
Let $\pi$ be a spherical representation of $G$.
For every $f\in \CS_{\rho}(G,K)$, $\vphi\in \CC(\pi)$ the integral
\begin{align*}
Z(s,f,\vphi) =\int_{G}f(g)\vphi(g)|\sig(g)|^{s+\frac{l}{2}}dg
\end{align*}
is a rational function in $q^{s}$, and the fractional ideal $I_{\pi} = \{Z(s,f,\vphi) | f\in \CS_{\rho}(G,K), \vphi \in \CC(\pi)  \}$ is equal to $L(s,\pi,\rho) \BC[q^{s},q^{-s}]$.
\end{thm}
The idea for the proof of Theorem \ref{gcd} is as follows. We notice that the function $f\in \CS_{\rho}(G,K)$ is bi-$K$-invariant. Following the proof of Proposition \ref{basicingredient}, we can actually assume that $\vphi$ is bi-$K$-invariant, which means that $\vphi$ is a scalar multiple of the zonal spherical function associated to $\pi$. Then, up to multiplying by a constant, the zeta integral $Z(s,f,\vphi)$ is equal to $\CS(f_{s+\frac{l}{2}})(c)$, where $c\in \wh{T}/W$ is the Satake parameter associated to $\pi$. Now Theorem \ref{gcd} follows from the definition of $\CS_{\rho}(G,K)$ and Remark \ref{rmkingredient}.

\begin{thm}\label{inv}
For any $f\in \CS_{\rho}(G,K)$, define the $\rho$-Fourier transform $\CF_{\rho}$ as in \cite{BK00} by the formula
\begin{align*}
\CF_{\rho}(f) = |\sig|^{-l-1}(\Phi_{\psi,\rho}*f^{\vee}).
\end{align*}
Then $\CF_{\rho}$ extends to a unitary operator on $L^{2}(G,K,|\sig|^{l+1}dg)$ and the space $\CS_{\rho}(G,K)$ is $\CF_{\rho}$-invariant.
\end{thm}
The idea for the proof of Theorem \ref{inv} is as follows. To show that $\CF_{\rho}$ extends to a unitary operator on $L^{2}(G,K,|\sig|^{l+1}dg)$, equivalently we need to show the following equality
$$
<\CF_{\rho}(f),\CF_{\rho}(h)>_{L^{2}(G,K,|\sig|^{l+1}dg)} = <f,h>_{L^{2}(G,K,|\sig|^{l+1}dg)}
$$
for any $f, h\in \CH(G,K)$, since the smooth compactly supported functions are dense in $L^{2}(G,K,|\sig|^{l+1}dg)$.

We first rewrite the integration as follows
\begin{align*}
<\CF_{\rho}(f),\CF_{\rho}(h)>_{L^{2}(G,K,|\sig|^{l+1}dg)} &=\CF_{\rho,l+1}(f)*\wb{\CF_{\rho}(h)}^{\vee}(e)
\\
<f,h>_{L^{2}(G,K,|\sig|^{l+1}dg)} &= \wb{h}_{l+1}*f^{\vee}(e).
\end{align*}
Then as in the proof of Proposition \ref{unitaryprop} we can show that after the Satake transform, the functions $\CF_{\rho,l+1}(f)*\wb{\CF_{\rho}(h)}^{\vee}$ and
$\wb{h}_{l+1}*f^{\vee}$ are equal to each other as a rational function on $\wh{T}/W$. Hence we get the first part of Theorem \ref{inv}. To show that the space $\CS_{\rho}(G,K)$ is $\CF_{\rho}$-invariant, we show that $\CF_{\rho}(\CS_{\rho}(G,K))$ and $\CS_{\rho}(G,K)$ have the same image under Satake transform.

\begin{thm}\label{thmfixed}
$\CF_{\rho}(1_{\rho,-\frac{l}{2}}) = 1_{\rho,-\frac{l}{2}}$.
\end{thm}
The idea for the proof of Theorem \ref{thmfixed} follows from the direct computation of the Satake transform of $\CF_{\rho}(1_{\rho,-\frac{l}{2}}) $ and $1_{\rho,-\frac{l}{2}}$. We show that they coincide with each other after Satake transform as a rational function on $\wh{T}/W$, from which we deduce that they are equal to each other.

The detailed proof of the theorems are given in Section \ref{proofpadic}.
\newline

\paragraph{\textbf{Archimedean Case:}}
We give a construction of $\Phi^{K}_{\psi,\rho}$ using the \emph{spherical Plancherel transform}.
More precisely,
\begin{defin}
We define $\Phi^{K}_{\psi,\rho,s} = 1_{\rho,1+s+\frac{l}{2}}*\CH^{-1}(\frac{1}{L(-s-\frac{l}{2},\pi,\rho^{\vee})})$.
Here $\CH$ is the \emph{spherical Plancherel transform}.
\end{defin}

Parallel to the $p$-adic case, we can verify that Conjecture \ref{fixed} holds through showing that $\CF_{\rho}(1_{\rho,-\frac{l}{2}})$ and $1_{\rho,-\frac{l}{2}}$ have the same image under spherical Plancherel transform $\CH$.

We also study asymptotic properties of $1_{\rho,s}$ and $\Phi^{K}_{\psi,\rho,s}$. We let $S^{p}(K\bs G/K)$ be the $L^{p}$-Harish-Chandra Schwartz space, where $0<p\leq 2$ is any real number. Then we can prove the following theorem.
\begin{thm}
\begin{enumerate}
\item If $F\cong \BR$, and $\Re(s)$ satisfies the following inequality
\begin{align*}
\Re(s)> \max\{ \vpi_{k}(\mu) | \quad 1\leq k\leq n,\mu \in C^{\veps\rho_{B}} \},
\end{align*}

or
\item
If $F\cong \BC$, and $\Re(s)$ satisfies the following inequality
\begin{align*}
\Re(s)> \max\{ \frac{\vpi_{k}(\mu)}{2} | \quad 1\leq k\leq n,\mu \in C^{\veps\rho_{B}} \},
\end{align*}
\end{enumerate}
then the function $1_{\rho,s}$ lies in $S^{p}(K\bs G/K)$.
\end{thm}
Here $\{\vpi_{k}\}_{k=1}^{n}$ are the weights of the reprentation $\rho$, $\veps = \frac{2}{p}- 1$, and $C^{\veps\rho_{B}}$ is the convex hull in $\Fa^{*}$ generated by elements $W\cdot\veps \rho_{B}$.

\begin{thm}
\begin{enumerate}
\item If $F\cong \BR$, and $\Re(s)$ satisfies the following inequality
\begin{align*}
\Re(s)> -1-\frac{l}{2} + \max\{ \vpi_{k}(\mu) | \quad 1\leq k\leq n,\mu \in C^{\veps\rho_{B}} \},
\end{align*}

or

\item If $F\cong \BC$, and $\Re(s)$ satisfies the following inequality
\begin{align*}
\Re(s)> -\frac{1}{2}-\frac{l}{4} + \max\{ \frac{\vpi_{k}(\mu)}{2} | \quad 1\leq k\leq n,\mu \in C^{\veps\rho_{B}} \},
\end{align*}
\end{enumerate}
then the function
$\Phi^{K}_{\psi,\rho,s}$ lies in $S^{p}(K\bs G/K)$.
\end{thm}

The idea for proving the asymptotic theorems is based on several asymptotic estimations for classical $\Gam$-function and its derivatives, which are recalled and proved in the beginning of Section \ref{archiproof}.

The details are presented in Section \ref{archiproof} and Section \ref{archiproofcplx}.

Jayce Getz \cite{getz2015nonabelian} also has similar descriptions for $\CF_{\rho}(f)$, where $f$ lies in $C^{\infty}_{c}(G,K)$. His description of the Fourier transform uses the relation between $\CF_{\rho}$ and the standard one on $\GL(n)$ also via \emph{spherical Plancherel transform}, in which the $\rho$-Fourier transform is not written as an explicit kernel function. Using the functional equation, one can observe that his definition coincides with our definition of $\CF_{\rho}$. On the other hand, using the explicit estimation for the kernel function $\Phi^{K}_{\psi,\rho,s}$, we find that our domain for the Fourier transform $\CF_{\rho}$ is bigger than $C^{\infty}_{c}(G,K)$. For instance, we can take Fourier transform for the basic function $1_{\rho,s}$.
\newline

\paragraph{\textbf{Organization of Paper}}
In Section \ref{sec:satake}, we have a quick review of the Satake isomorphism. In Section \ref{sec:structure}, we give a description of the structure of $\CS_{\std}(G,K)$, which are the restriction of functions in $C^{\infty}_{c}(\RM_{n},K)$ to $G=\GL(n)$. In Section \ref{sec:basicfun}, we briefly review the theory of basic functions. In Section \ref{proofpadic} we prove the unramified part of the conjectures mentioned in the introduction.

In Section \ref{sec:spplancherel}, we review the theory of spherical plancherel transform. In Section \ref{sec:corrreal} and \ref{sec:corrcplx}, we review the Langlands classification and Langlands correspondence of spherical representations for $\GL_{n}(\BR)$ and $\GL_{n}(\BC)$. From the Langlands classification and Langlands correspondence, we obtain the explicit formula of local $L$-factors. In Section \ref{archiproof} and \ref{archiproofcplx} we prove asymptotic properties of $1_{\rho,s}$ and $\Phi^{K}_{\psi,\rho,s}$, from which we can deduce the theorems mentioned in the introduction.
\newline

\paragraph{\textbf{Acknowledgement }}
I would like to express my sincere gratitude to my advisor Prof. Dihua Jiang, who encourages me continuously, and gives me helpful advice when I was writing down the paper. I would also like to thank Fangyang Tian for discussing the proof of Lemma \ref{assumpweight}, Chen Wan for reading the manuscript carefully, and Jorin Schug for correcting some grammars. Finally, I would like to thank the anonymous referee for several useful comments.

\section{$p$-Adic Case}
The theory of spherical functions and spherical representations for $p$-adic groups are developed by I. Satake in \cite{satake63}. In particular, Satake proves that under the Satake transform $\CS$, the spherical Hecke algebra $\CH(G,K)$ is isomorphic to $\BC[\wh{T}/W]$, which nowadays is called the \emph{Satake isomorphism}.

On the other hand, $\CH(G,K)$ is contained in the conjectural function space $\CS_{\rho}(G,K)$ as a proper subspace. In order to obtain a similar description for $\CS_{\rho}(G,K)$, we need to extend the Satake isomorphism to $\CS_{\rho}(G,K)$. This is achieved in \cite[Proposition~2.3.2]{wwlibasic}. For the basic function $1_{\rho,-\frac{l}{2}}$, although it is not compactly supported on $G$, it is compactly supported on the sets $\{ g\in G |\quad |\sig(g)| = q^{-n} \}_{n\in \BZ}$. For different $n$, the sets are disjoint. This means that the function $1_{\rho,-\frac{l}{2}}$ is \emph{almost compactly supported} as defined in \cite[Definition~2.3.1]{wwlibasic}. In particular we can apply the Satake isomorphism to $1_{\rho,-\frac{l}{2}}$.

Using the Satake isomorphism, we will give a definition of $\CS_{\rho}(G,K)$ and $\Phi^{K}_{\psi,\rho,s}$, and we can verify several conjectures in this case as mentioned in the introduction.

\subsection{Satake Isomorphism}\label{sec:satake}
In this section, we review the \emph{Satake isomorphism}. The main references are \cite{cartier79}, \cite{grosssatake} and \cite{satake63}.

First we give the definition of Satake transform.
\begin{defin}[Satake Transform]
For $f\in \CH(G,K)$, the function $\CS(f)$ is defined to be
\begin{align*}
\CS(f)(t) = \del_{B}^{\frac{1}{2}}(t)\int_{N}f(tn)dn.
\end{align*}
\end{defin}

In \cite{satake63}, Satake proves the fact that $\CS$ is an algebra isomorphism from $\CH(G,K)$ to $\CH(T,T_{K})^{W}$, where both algebras are equipped with convolution structure.

Using the canonical $W$-equivariant isomorphisms
\begin{align*}
T/T_{K} \cong X_{*}(T)\cong X^{*}(\wh{T}),
\end{align*}
we have
\begin{align*}
\CH(T,T_{K})^{W}\cong \BC[X_{*}(T)]^{W}\cong \BC[X^{*}(\wh{T})]^{W}.
\end{align*}
Since $\BC[X^{*}(\wh{T})]$ consists of $\BC$-linear combinations of algebraic characters of $\wh{T}$, it can naturally be identified with algebraic functions on $\wh{T}$. Therefore $\BC[X^{*}(\wh{T})]^{W}\cong \BC[\wh{T}/W]$. Sometimes we abuse the notation of Satake transform $\CS$ with the image identified with $\BC[\wh{T}/W]$.

\subsection{Structure of $\CS_{\std}(G,K)$}\label{sec:structure}
In this section, we review the theory of the zeta integrals for the standard $L$-function of $\GL(n)$ over a non-archimedean local field following the approach of Godement-Jacquet. The main references are \cite{gjzeta} and
\cite{jacquet79}. In the end we give a description of the structure of $\CS_{\std}(G,K)$.

In \cite{gjzeta}, Godement and Jacquet established the theory of standard $L$-function for multiplicative group of central simple algebras following the approach of \cite{tatethesis}. For our purpose, we only focus on
$G=\GL(n)$, though the story for multiplicative group of central simple algebras is almost the same.

Let $(\pi,V)$ be an admissible representation of $G$ with smooth admissible contragredient dual $(\pi^{\vee},\wt{V})$. Let
\begin{align*}
<,>: \wt{V}\times V &\to \BC\\
(\wt{v},v) &\to <\wt{v},v>
\end{align*}
be the canonical linear pairing between $\wt{V}$ and $V$.

Let $\CC(\pi)$ be the $\BC$-linear span of the following functions
\begin{align*}
\pi_{\wt{v},v}: g\to <\wt{v},\pi(g)v>, v\in V, \wt{v}\in \wt{V}.
\end{align*}
Elements in $\CC(\pi)$ are called the matrix coefficients of $\pi$.

By the admissibility of $\pi$, the smooth contragredient of $\pi^{\vee}$ is canonically isomorphic to $\pi$. It follows that for any $\vphi \in \CC(\pi)$, the function
\begin{align*}
\vphi^{\vee}(g)=\vphi(g^{-1})
\end{align*}
is a matrix coefficient of $\pi^{\vee}$.

Let $\RM_{n}(F)$ be the space of $n\times n$ matrices over $F$. Let $C^{\infty}_{c}(\RM_{n})$ be the space of smooth compactly supported functions on $\RM_{n}(F)$.

For $\vphi \in \CC(\pi)$, $f\in C^{\infty}_{c}(\RM_{n})$, $s\in \BC$, one set
\begin{align}\label{zetint}
Z(s,f,\vphi) = \int_{G}f(g)\vphi(g)|\det g|^{s+\frac{n-1}{2}}d^{\times}g.
\end{align}

In \cite{gjzeta}, the following proposition was proved.

\begin{pro}\cite[Proposition~(1.2)]{jacquet79}\label{gjmain}
Suppose that $\pi$ is an irreducible and admissible representation of $G$, then

\begin{enumerate}
\item There exists $s_{0}\in \BC$ such that the integral (\ref{zetint}) converges absolutely for $\Re (s) >\Re(s_{0})$.

\item The integral (\ref{zetint}) is given by a rational function in $q^{-s}$, where $q$ is the cardinality of the residue field of $F$. Moreover, the family of rational functions in $q^{-s}$
\begin{align*}
I(\pi)= \{Z(s,f,\vphi) | \quad f\in C^{\infty}_{c}(\RM_{n}), \vphi\in \CC(\pi) \}
\end{align*}
admits a common denominator which does not depend on $f$ or $\vphi$.

\item Let $\psi \neq 1$ be an additive character of $F$. There exists a rational function $\gam(s,\pi,\psi)$ such that for any $\vphi\in \CC(\pi)$ and $f\in C^{\infty}_{c}(\RM_{n})$, we have the following functional equation
\begin{align}\label{funeq}
Z(1-s,\CF(f),\vphi^{\vee}) = \gam(s,\pi,\psi)Z(s,f,\vphi),
\end{align}
where $\CF(f)$ is the Fourier transform of $f$ w.r.t. $\psi$
\begin{align*}
\CF(f)(x) = \int_{\RM_{n}}f(y)\psi(\tr(yx))dy.
\end{align*}
Here we choose $dy$ to be the self-dual Haar measure on $\RM_{n}(F)$, in the sense that
\begin{align*}
\CF(\CF (f))(x) = f(-x).
\end{align*}
\end{enumerate}
\end{pro}

Now we prove the claim in the introduction, that the space
$\CS_{\std}(G,K)$, which consists of the restriction to $G=\GL(n)$ of bi-$K$-invariant functions in the space $C^{\infty}_{c}(\RM_{n}(F))$, has the following simple expression
$$
1_{\RM_{n}(\CO)}*\CH(G,K) = 1_{\std,-\frac{n-1}{2}}*\CH(G,K).
$$

\begin{pro}\label{basicingredient}
$\CS_{\std}(G,K) = 1_{\RM_{n}(\CO_{F})}*\CH(G,K)$.
\end{pro}
\begin{proof}
Let $\pi = \pi_{c}$ be a spherical representation of $G$ with Satake parameter $c\in \wh{T}/W$. By Proposition \ref{gjmain}
\begin{align*}
\{ \frac{Z(s,f,\vphi_{\pi})}{L(s,\pi)}| \quad  f\in C^{\infty}_{c}(\RM_{n}(F)), \vphi_{\pi}\in \CC(\pi)  \}= \BC[q^{-s},q^{s}].
\end{align*}
Now for any matrix coefficient $\vphi_{\pi}(g) = <\wt{v},\pi(g)v>$ in $\CC(\pi)$, there exists finitely many constant numbers $c_{i}$ in $\BC$, $h^{i}_{0}$ and $g^{i}_{0}$ $(1\leq i\leq n)$ in $G$, such that $\vphi_{\pi}(g)  = \sum_{i=1}^{n}c_{i}\Gam_{\chi}(h^{i}_{0}gg^{i}_{0})$, where $\Gam_{\chi}$ is the zonal spherical function associated to $\pi$. Therefore up to translation and scaling, we can assume that our $\vphi_{\pi}$ is just the zonal spherical function $\Gam_{\chi}$. Moreover
\begin{align*}
Z(s,f,\Gam_{\chi}) &= \int_{G}f(g)\Gam_{\chi}(g)|\det g|^{s+\frac{n-1}{2}}dg
\\&= f|\det|^{s+\frac{n-1}{2}}*\Gam_{\chi}^{\vee}(e),
\end{align*}
and by the fact that $G$ is unimodular
\begin{align*}
Z(s,f,\Gam_{\chi}) &= \int_{G}f(g^{-1})\Gam_{\chi}(g^{-1})|\det g^{-1}|^{s+\frac{n-1}{2}}dg
\\ &= \Gam_{\chi}^{\vee}*f|\det|^{s+\frac{n-1}{2}}(e).
\end{align*}
Since $\Gam^{\vee}_{\chi}$ is bi-$K$-invariant, we can assume that $f$ is bi-$K$-invariant as well.
It follows that Proposition \ref{gjmain} in the spherical case can be restated as
\begin{align*}
\{ Z(s,f,\Gam_{\chi}) |\quad f\in \CS_{\std}(G,K) \}   =L(s,\pi)  \BC[q^{-s},q^{s}].
\end{align*}

Now we notice that $Z(s,f,\Gam_{\chi}) = \CS(f|\det|^{s+\frac{n-1}{2}})(c)$. If we let $\Gam_{\chi,s}$ be the zonal spherical function associated to $\pi_{s} = \pi |\det|^{s}$, then we have $Z(s,f,\Gam_{\chi}) = Z(0,f,\Gam_{\chi,s})$, and
\begin{align*}
Z(0,f,\Gam_{\chi,s}) = \CS(f|\det|^{\frac{n-1}{2}})(c\cdot q^{-s})  =\CS(f)(c\cdot q^{-s-\frac{n-1}{2}}),
\end{align*}
where $c\cdot q^{-s}$ is the Satake parameter of $\pi_{c,s} = \pi_{c}|\det |^{s}$.

Therefore
\begin{align*}
Z(s,\CH(G,K),\Gam_{\chi}) = \CS(\CH(G,K))(c\cdot q^{-s-\frac{n-1}{2}}) = \BC[\wh{T}/W](c\cdot q^{-s-\frac{n-1}{2}}).
\end{align*}
The space $\BC[\wh{T}/W](c\cdot q^{-s-\frac{n-1}{2}})$ is contained in $\BC[q^{s},q^{-s}]$ naturally.

On the other hand, the space
\begin{align*}
\BC[\wh{T}/W](c\cdot q^{-s-\frac{n-1}{2}}) = \{ Z(s,f,\Gam_{\chi}) |\quad f\in \CH(G,K) \}
\end{align*}
can be identified with
\begin{align*}
\{ Z(s,f,\vphi_{\pi}) |\quad f\in C^{\infty}_{c}(G), \vphi_{\pi}\in \CC(\pi) \}
\end{align*}
using the same argument as the beginning of the proof. Moreover, the space $\{ Z(s,f,\vphi_{\pi}) |\quad f\in C^{\infty}_{c}(G), \vphi_{\pi}\in \CC(\pi) \}$ is a fractional ideal of $\BC[q^{s},q^{-s}]$ containing the constants, it follows that $\{ Z(s,f,\vphi_{\pi}) |\quad f\in C^{\infty}_{c}(G), \vphi_{\pi}\in \CC(\pi) \} = \BC[q^{s}, q^{-s}]$, and we have proved that  $\BC[\wh{T}/W](c\cdot q^{-s-\frac{n-1}{2}}) = \BC[q^{-s},q^{s}]$.
Therefore we get
\begin{align*}
\frac{Z(s,f,\Gam_{\chi})}{L(s,\pi)} \in \BC[\wh{T}/W](c\cdot q^{-s-\frac{n-1}{2}}), \quad f\in \CS_{\std}(G,K).
\end{align*}
Letting $s = \frac{1-n}{2}$, we get
\begin{align*}
\CS(f) \in \CS(1_{\RM_{n}(\CO_{F})})\BC[\wh{T}/W] = \CS(1_{\RM_{n}(\CO_{F})}*\CH(G,K)).
\end{align*}
From this we get $\CS_{\std}(G,K)\subset 1_{\RM_{n}(\CO_{F})}*\CH(G,K)$, and therefore we have proved the equality
\begin{align*}
\CS_{\std}(G,K) = 1_{\RM_{n}(\CO_{F})}*\CH(G,K).
\end{align*}
\end{proof}

\begin{rmk}\label{rmkingredient}
Actually from the proof of Proposition \ref{basicingredient} we find that if a smooth bi-$K$-invariant function $f$  satisfies the condition
\begin{align*}
Z(s,f,\Gam_{\chi})\subset \BC[q^{s},q^{-s}], \quad \textrm{ for any unramified character $\chi$},
\end{align*}
 then the function $f$ lies in $\CH(G,K)$.
\end{rmk}

Theorem \ref{basicingredient} will be our basic ingredient for introducing the space $\CS_{\rho}(G,K)$.

\subsection{Unramified $L$-Factors and the Basic Function}\label{sec:basicfun}
In this section, we review basic results of \emph{basic function}. The main references are \cite{wwlibasic} and \cite{yiannisinverse}.

Let $\pi_{c}$ be the spherical representation of $G$ with Satake parameter $c\in \wh{T}/W$.

First we recall the definition of unramified local $L$-factor.
\begin{defin}
The unramified local $L$-factor attached to $\pi_{c}$ and $\rho$ is defined by
\begin{align*}
L(\pi_{c},\rho,X) = \det(1-\rho(c)X)^{-1},
\end{align*}
which is a rational function in $X$.
\end{defin}

The usual $L$-factors are obtained by specializing $X$, namely
\begin{align*}
L(s,\pi_{c},\rho) = L(\pi_{c},\rho,q^{-s}), s\in \BC.
\end{align*}

Then we recall the following identity.

\begin{lem}\cite[Proposition~43.5]{bumplietheory}
\begin{align*}
L(s,\pi_{c},\rho) = [ \sum_{i=0}^{n} (-1)^{i}\tr (\bigwedge^{i}\rho(c)) q^{-is}]^{-1}
=\sum_{k\geq 0}\tr(\Sym^{k}\rho(c))q^{-ks}.
\end{align*}
\end{lem}

Here we notice that, by assumption $\rho\circ \wh{\sig} $ can be identified with the standard embedding of $\BG_{m}$ into $\GL(V_{\rho})$ via $z\to z\Id$. Moreover, following the assumption in  \cite[Section~3.2]{wwlibasic}, the restriction of $\rho$ to the central torus is $z\to z\Id$, $z\in \BC$. Therefore we find that for all $s\in \BC$,
\begin{align*}
L(\pi_{c}\otimes |\sig|^{s}, \rho,X)
&=\det (1-\rho(c\cdot q^{-s})X)^{-1}
\\ = \det(1-\rho(c\cdot q^{-s}\Id)X)^{-1} &= \det(1-\rho(c)q^{-s}X)^{-1}
\\&= L(\pi_{c},\rho, q^{-s}X).
\end{align*}
Now to define the basic function $1_{\rho,s}$, we want to apply the inverse Satake isomorphism to $L(\pi_{c},\rho,X)$. But $L(\pi_{c},\rho,X)$ is a rational function rather than polynomial on $\wh{T}/W$, hence we need to analyze the support of the inverse Satake transform of $L(\pi_{c},\rho,X)$. Following \cite[Section~3.2]{wwlibasic}, we give an argument showing that the basic function is a formal sum of compactly supported functions on $G$ with disjoint support.

We recall the Kato-Lusztig formula for inverse Satake transform
\begin{thm}[\cite{inversekato}, \cite{inverselusztig}]\label{katolusztig}
For $\lam\in X_{*}(T)_{+}=X^{*}(\wh{T})_{+}$, let $V(\lam)$ be the irreducible representation of $\LG$ of highest weight $\lam$, then
\begin{align*}
\tr V(\lam) = \sum_{\mu\in X_{*}(T)_{+},\mu\leq \lam}q^{-<\rho_{B},\mu>}K_{\lam,\mu}(q^{-1})\CS(1_{K\mu(\vpi)K})
\end{align*}
as an element in $\CH(T,T_{K})^{W}$. Here the function $K_{\lam,\mu}$ is the \emph{Lusztig's $q$-analogue} of Kostant's partition function as mentioned in \cite[Section~2.2]{wwlibasic}.
\end{thm}
If we let mult($\Sym^{k}\rho:V(\lam)$) be the multiplicity of $V(\lam)$ in $\Sym^{k}\rho$, then
\begin{align*}
L(\pi_{c},\rho,X) = \sum_{k\geq 0}\sum_{\lam\in X_{*}(T)_{+}}\textrm{mult}(\Sym^{k}\rho:V(\lam))\tr V(\lam)(c)X^{k}.
\end{align*}

By the Kato-Lusztig formula, it equals
\begin{align*}
\sum_{k\geq 0}
\bigg\{ \sum_{\lam,\mu\in X_{*}(T)_{+},\mu\leq \lam} \mathrm{mult}(\Sym^{k}\rho: V(\lam))q^{-<\rho_{B},\mu>}
\\
K_{\lam,\mu}(q^{-1})\CS(1_{K\mu(\vpi)K})(c) \bigg\} X^{k}\\
=\sum_{\mu \in X_{*}(T)_{+}} \bigg\{
\sum_{k\geq 0}\sum_{\lam\in X_{*}(T)_{+}, \lam\geq \mu}
K_{\lam, \mu}(q^{-1})\mathrm{mult}(\Sym^{k}\rho:V(\lam))X^{k}
\bigg\}
\\
q^{-<\rho_{B},\mu>}
\CS(1_{K\mu(\vpi)K})(c).
\end{align*}

Here we observe that each weight $\nu$ of $\Sym^{k}\rho$ satisfies
$\sig(\nu) = k$, where $k$ is identified with the character of $\BG_{m}:z\to z^{k}$.  Thus for each $\mu\in X_{*}(T)_{+}$, the inner sum can be taken over $k=\sig(\mu)$.

For $\mu \in X_{*}(T)_{+}$, we set
\begin{align}\label{newcoeff}
c_{\mu}(q) =
\sum_{\lam\in X_{*}(T)_{+}, \lam\geq \mu}
K_{\lam, \mu}(q^{-1})\mathrm{mult}(\Sym^{k}\rho:V(\lam)),
\end{align}
if $\sig(\mu)\geq 0$, and $0$ otherwise.

We have to justify the rearrangement of sums. Given $\mu$ with $\sig(\mu)= k \geq 0$, the expression (\ref{newcoeff}) is a finite sum over those $\lam$ with
$\sig(\lam)= k$ as explained above, and hence is well-defined.
On the other hand, given $k\geq 0$, there are only finitely many $V(\lam)$ that appear in $\Sym^{k}\rho$. Thus only finitely many $\mu\in X_{*}(T)_{+}$ with
$\sig(\mu) = k$ and $c_{\mu}(q) \neq 0$.
To sum up, we arrive at the following equation in $\BC[[X]]$
\begin{align*}
L(\pi_{c},\rho,X) = \sum_{\mu\in X_{*}(T)_{+}}c_{\mu}(q)q^{-<\rho_{B},\mu>}
\CS(1_{K\mu(\vpi)K})(c)X^{\sig(\mu)}.
\end{align*}

Now we define the function $\vphi_{\rho,X}:T(F)/T_{K}\to \BC[X]$ by
\begin{align*}
\vphi_{\rho,X} = \sum_{\mu\in X_{*}(T)_{+}}c_{\mu}(q)q^{-<\rho_{B},\mu>}
\CS(1_{K\mu(\vpi)K})X^{\sig(\mu)}.
\end{align*}
By previous argument we find that for fixed $k$,
\begin{align*}
 \sum_{\lam,\mu\in X_{*}(T)_{+},\mu\leq \lam} \mathrm{mult}(\Sym^{k}\rho: V(\lam))q^{-<\rho_{B},\mu>}
K_{\lam,\mu}(q^{-1})\CS(1_{K\mu(\vpi)K})(c)X^{k}
\end{align*}
lies in $\CH(T,T_{K})^{W}$.
Hence
$\vphi_{\rho,X}$ is a formal sum of functions in $\CH(T,T_{K})^{W}$.

\begin{defin}
Define the \emph{basic function} $1_{\rho,X}$ as a formal sum of functions, each is supported on
$\{\mu\in X_{*}(T)_{+} |\sig(\mu) = k\}$ for some $k\geq 0$ lying in $\CH(G,K)$ as
\begin{align*}
1_{\rho,X} = \sum_{\mu\in X_{*}(T)_{+}}c_{\mu}(q)q^{-<\rho_{B},\mu>}1_{K\mu(\vpi)K}X^{\sig(\mu)}.
\end{align*}
\end{defin}
One may specialize the variable $X$. Define $1_{\rho,s}$ as the specialization at $X=q^{-s}$. Then
$$1_{\rho,s} = 1_{\rho}|\sig|^{s}.$$

In \cite{wwlibasic}, several analytical properties of $1_{\rho,s}$ has been proved.
By definition, we have $\CS(1_{\rho,X}) = \varphi_{\rho,X}$. Let $c\in \wh{T}/W$ and $\pi_{c}$ be the $K$-unramified irreducible representation with Satake parameter $c$. Let $V_{c}$ denote the underlying $\BC$-vector space of
$\pi_{c}$. Then
\begin{align*}
\varphi_{\rho,X}(c) = L(\pi_{c},\rho,X)
\end{align*}
is a rational function in $c\in \wh{T}/W$.
For $\Re(s)$ sufficiently large with respect to $c$, the operator $\pi_{c}(1_{\rho,s}):V_{c}\to V_{c}$ and its trace are well-defined and
\begin{align*}
\tr(1_{\rho,s}|V_{c}) = L(s,\pi_{c},\rho).
\end{align*}
Moreover, it is shown in \cite{wwlibasic} that the coefficient $c_{\mu}(q)$ is of polynomial growth w.r.t $\mu$, and the integrability of $1_{\rho,s}$ when $\Re(s)$ is sufficiently large has also been demonstrated. We refer the reader to the paper \cite{wwlibasic} for further details.

\subsection{Construction of $\CS_{\rho}(G,K)$ and $\CF_{\rho}$}\label{proofpadic}
In this section, we give a definition of the space $\CS_{\rho}(G,K)$ and construct the spherical component of the operator $\CF_{\rho}$ using the inverse Satake transform.

The definition is motivated from the structure of $\CS_{\std}(G,K)$ as shown in Proposition \ref{basicingredient}.

\begin{defin}
We define the function space $\CS_{\rho}(G,K)$ to be $1_{\rho,-\frac{l}{2}}*\CH(G,K)$.
\end{defin}

By our definition of $\CS_{\rho}(G,K)$, the spherical part of Conjecture \ref{zetaintegral} holds automatically.
Moreover, following the proof of Proposition \ref{basicingredient}, we find that
\begin{align*}
\{ Z(s,f,\Gam_{\chi}) | \quad f\in \CH(G,K) \} = \BC[\wh{T}/W](c\cdot q^{-s-\frac{l}{2}}) = \BC[q^{s},q^{-s}]
\end{align*}
for any spherical representation $\pi_{c}$ with Satake parameter $c\in \wh{T}/W$. From the proof of Proposition \ref{basicingredient}, we realize that $\CS_{\rho}(G,K)$ is the largest subspace of $C^{\infty}(G,K)$ satisfying the spherical part of Conjecture \ref{zetaintegral}.
In other words, the following theorem holds.
\begin{thm}
Let $\pi$ be a spherical representation of $G$.
For every $f\in \CS_{\rho}(G,K)$, $\vphi\in \CC(\pi)$ the integral
\begin{align*}
Z(s,f,\vphi) =\int_{G}f(g)\vphi(g)|\sig(g)|^{s+\frac{l}{2}}dg
\end{align*}
is a rational function in $q^{s}$, and $I_{\pi} = \{Z(s,f,\vphi) |\quad f\in \CS_{\rho}(G,K), \vphi \in \CC(\pi)  \}  = L(s,\pi,\rho) \BC[q^{s},q^{-s}]$.
\end{thm}

Using our definition, we can also show the following

\begin{lem}
$\CS_{\rho}(G,K)$ contains $\CH(G,K)$.
\end{lem}
\begin{proof}
By Satake isomorphism, the space $\CS(1_{\rho,-\frac{l}{2}}*\CH(G,K))$ as rational functions on $c\in \wh{T}/W$ is equal to $L(-\frac{l}{2},\pi_{c},\rho)\BC[\wh{T}/W]$, which contains $\BC[\wh{T}/W]$. Applying inverse Satake transform and we get the lemma.
\end{proof}

Then we give our definition of the spherical component of the kernel $\Phi_{\psi,\rho}^{K}$.
Before that we show how to derive the relation between $\gam(s,\pi,\rho,\psi)$ and $\Phi_{\psi, \rho}$ from the conjectural functional equation
\begin{align*}
Z(1-s,\CF_{\rho}(f), \vphi^{\vee}) =\gam(s,\pi,\rho,\psi)Z(s,f,\vphi),
\end{align*}
where
\begin{align*}
Z(s,f,\vphi) =\int_{G}f(g)\vphi(g)|\sig (g)|^{s+\frac{l}{2}}dg
\end{align*}
and $\vphi(g) = <\wt{v}, \pi(g)v>$ lies in $\CC(\pi)$.

Since the analytical property of $\Phi_{\psi,\rho}$ is still conjectural, the proof of the following lemma is purely formal. But later when restricting to the spherical component, we can make it to be rigorous.

\begin{lem}\label{distributiongamma}
For any irreducible admissible representation $\pi$ of $G$
\begin{align*}
\pi(\Phi_{\psi,\rho,s}) = \gam(-s-\frac{l}{2},\pi^{\vee},\rho,\psi)\Id.
\end{align*}
\end{lem}
\begin{proof}
As conjectured in \cite{BK00}, the function $\CF_{\rho}(f)$ is defined to be
\begin{align*}
|\sig|^{-l-1}(\Phi_{\psi,\rho}*f^{\vee}).
\end{align*}
We plug the formula into the functional equation, and get
\begin{align}\label{1functional}
<\wt{v},Z(1-s,|\sig|^{-l-1}(\Phi_{\psi,\rho}*f^{\vee}),\pi^{\vee})v>
\\=
\gam(s,\pi,\rho,\psi)<\wt{v},Z(s,f,\pi)v>.
\end{align}
Here $Z(s,f,\pi)$ is defined to be the operator $\int_{G}f(g)\pi(g)|\sig (g)|^{\frac{l}{2}}dg$ whenever $\Re(s)$ is sufficiently large.
For the left hand side of (\ref{1functional}), we can further simplify it to be
\begin{align*}
Z(1-s,|\sig|^{-l-1}(\Phi_{\psi,\rho}*f^{\vee}),\pi^{\vee}) &= Z(-s-l,\Phi_{\psi,\rho}*f^{\vee},\pi^{\vee})
\\&=
(\pi^{\vee})_{-s-\frac{l}{2}}(\Phi_{\psi,\rho})
(\pi^{\vee})_{-s-\frac{l}{2}}(f^{\vee}).
\end{align*}
Then the conjectural identity can be simplified to be
\begin{align*}
<\wt{v},(\pi^{\vee})_{-s-\frac{l}{2}}(\Phi_{\psi,\rho})
(\pi^{\vee})_{-s-\frac{l}{2}}(f^{\vee})v> = \gam(s,\pi,\rho,\psi) <\wt{v},\pi_{s+\frac{l}{2}}(f)v>.
\end{align*}

Now by assumption, $\Phi_{\psi,\rho}$ is a $G$-stable distribution, therefore it should be conjugation-invariant. Then by Schur's lemma the operator $(\pi^{\vee})_{-s-\frac{l}{2}}(\Phi_{\psi,\rho})$ should act as a scalar $c(s)$. Hence the identity can be further simplified as
\begin{align*}
c(s)<\wt{v},(\pi^{\vee})_{-s-\frac{l}{2}}(f^{\vee})v> = \gam(s,\pi,\rho,\psi)<\wt{v},\pi_{s+\frac{l}{2}}(f)v>.
\end{align*}
Now we arrive at the equality
\begin{align*}
c(s)Z(-s-l , f^{\vee},\vphi^{\vee}) = \gam(s,\pi,\rho,\psi) Z(s,f,\vphi).
\end{align*}
Using the identity $Z(-s-l, f^{\vee}, \vphi^{\vee})  = Z(s,f,\vphi)$, we get
\begin{align*}
c(s) = \gam(s,\pi,\rho,\psi).
\end{align*}
In other words, we obtain
\begin{align*}
(\pi^{\vee})_{-s-\frac{l}{2}}(\Phi_{\psi,\rho}) = \gam(s,\pi,\rho,\psi)\Id,
\end{align*}
which is equivalent to the desired relation
\begin{align*}
\pi (\Phi_{\psi,\rho,s})  = \gam(-s-\frac{l}{2},\pi^{\vee},\rho,\psi) \Id.
\end{align*}
\end{proof}

Now we restrict our representation $\pi$ to be a spherical representation.

By the definition of $\gam$-factor in spherical case,
we know that
\begin{align*}
\gam(s,\pi,\rho,\psi) =\veps(s,\pi,\rho,\psi)\frac{L(1-s,\pi^{\vee},\rho)}{L(s,\pi,\rho)}.
\end{align*}
Since we assume that $\psi$ is self-dual, which means that $\psi$ has level $0$. By the computations in \cite{gjzeta} we know that $\veps(s,\pi,\rho,\psi) = 1$ when $\rho$ is the standard representation of $\GL(n)$. In order to be consistent with the functoriality for general $\rho$, which means that $\veps(s,\pi,\rho,\psi) = \veps(s,\rho(\pi),\psi)$, where $\rho(\pi)$ is the functorial lifting of $\pi$ along $\rho$, we can just let $\veps(s,\pi,\rho,\psi) = 1$ for general $\rho$ whenever $\psi$ is of level $0$.

Therefore $\gam(s,\pi,\rho,\psi)$ can be simplified as
\begin{align*}
\gam(s,\pi,\rho,\psi) =\frac{L(1-s,\pi^{\vee},\rho)}{L(s,\pi,\rho)}.
\end{align*}
If we assume that the spherical representation $\pi$ has Satake parameter $c\in \wh{T}/W$, then $\pi^{\vee}$ has Satake parameter $c^{-1}\in \wh{T}/W$. For convenience, we write $\pi_{c}$ to mean that the spherical representation has Satake parameter $c\in \wh{T}/W$.

Using the definition of unramified $L$-factor, we find that
\begin{align*}
L(s,\pi,\rho) &= \det(1-\rho(c)q^{-s})^{-1},
\\
 L(1-s,\pi^{\vee},\rho) &=\det(1-\rho(c^{-1})q^{1-s})^{-1}.
\end{align*}

On the other hand, we know that
\begin{align*}
\det(1-\rho(c^{-1})q^{1-s}) = \det(1-\rho^{\vee}(c)q^{1-s}),
\end{align*}
where $\rho^{\vee}$ is the contragredient of $\rho$.

It follows that $\gam(s,\pi,\rho,\psi)$ can be further simplified to be
\begin{align*}
\gam(s,\pi,\rho,\psi) =\frac{L(1-s,\pi,\rho^{\vee})}{L(s,\pi,\rho)}.
\end{align*}

Now by previous discussion, we know that
\begin{align*}
\pi(\Phi_{\psi,\rho,s}) = \gam(-s-\frac{l}{2},\pi^{\vee},\rho,\psi)\Id.
\end{align*}

Using the inverse Satake isomorphism, we get the spherical component of the distribution $\Phi_{\psi,\rho,s}$, which we denote by $\Phi_{\psi,\rho,s}^{K}$,
\begin{align*}
\Phi_{\psi,\rho,s}^{K}&= \CS^{-1}(\gam(-s-\frac{l}{2},\pi^{\vee},\rho,\psi))
\\
&=\CS^{-1}(L(1+s+\frac{l}{2},\pi^{\vee},\rho^{\vee})) *\CS^{-1} (\frac{1}{L(-s-\frac{l}{2},\pi^{\vee},\rho)}).
\end{align*}
Since $L(1+s+\frac{l}{2},\pi^{\vee},\rho^{\vee}) = L(1+s+\frac{l}{2},\pi,(\rho^{\vee})^{\vee}) = L(1+s+\frac{l}{2},\pi,\rho)$,
and $L(-s-\frac{l}{2},\pi^{\vee},\rho) = L(-s-\frac{l}{2},\pi,\rho^{\vee})$, we get
\begin{align*}
\Phi^{K}_{\psi,\rho,s}&=\CS^{-1}(L(1+s+\frac{l}{2},\pi,\rho))*\CS^{-1}(\frac{1}{L(-s-\frac{l}{2},\pi,\rho^{\vee})})\\
&=1_{\rho,1+s+\frac{l}{2}}*\CS^{-1}(\frac{1}{L(-s-\frac{l}{2},\pi,\rho^{\vee})}).
\end{align*}

\begin{rmk}\label{realvalued}
We notice that for a fixed $s\in \BC$, as a function in Satake parameter $c\in \wh{T}/W$,
$\frac{1}{L(-s-\frac{l}{2},\pi,\rho^{\vee})}=\frac{1}{L(-s-\frac{l}{2},\pi_{c},\rho^{\vee})}$ lies in $\BC[\wh{T}/W]$, therefore $\CS^{-1}(\frac{1}{L(-s-\frac{l}{2},\pi,\rho^{\vee})})$ lies in
$\CH(G,K)$. We also notice that the spectral property of $\Phi_{\psi,\rho,s}^{K}$ is really determined by the basic function $1_{\rho,s}$. On the other hand, we find that when writing the function $\Phi^{K}_{\psi,\rho}$ as expansion via basis $\{ 1_{K\lam K} \}_{\lam \in X_{*}(T)_{+}}$, all its coefficients are real numbers, from which we deduce that the complex conjugate of $\Phi^{K}_{\psi,\rho}$, which we denote as
$\wb{\Phi^{K}_{\psi,\rho}}$ is equal to  $\Phi^{K}_{\psi,\rho}$. This will be useful for proving Proposition \ref{unitaryprop}.
\end{rmk}

By construction, our definition of $\Phi_{\psi,\rho}^{K}$ does give us the functional equation
\begin{align*}
Z(1-s, \CF_{\rho}(f),\vphi) = \gam(s,\pi,\rho,\psi)Z(s,f,\vphi) , \quad f\in \CS_{\rho}(G,K),
\end{align*}
where $\CF_{\rho}(f)  = |\sig|^{-l-1}(\Phi_{\psi,\rho}^{K}*f^{\vee})$.

\begin{pro}\label{fixbasic}
Conjecture \ref{fixed} holds, i.e. $\CF_{\rho}$ sends basic function $1_{\rho,-\frac{l}{2}}$ to $1_{\rho,-\frac{l}{2}}$.
\end{pro}
\begin{proof}
By definition
\begin{align*}
\CF_{\rho}(1_{\rho,-\frac{l}{2}})(g) = |\sig (g)|^{-l-1}(\Phi_{\psi,\rho}^{K}*1_{\rho,-\frac{l}{2}}^{\vee})(g)
=|\sig(g)|^{-l-1}(\Phi_{\psi,\rho}^{K}*(1^{\vee}_{\rho})_{\frac{l}{2}})(g).
\end{align*}
Applying the Satake isomorphism to the function $\Phi_{\psi,\rho}^{K}*(1_{\rho}^{\vee})_{\frac{l}{2}}$, one gets that as rational function on $\wh{T}/W$
\begin{align*}
\CS(\Phi_{\psi,\rho}^{K}*(1_{\rho}^{\vee})_{\frac{l}{2}})(c) = \CS(\Phi_{\psi,\rho}^{K})(c)\CS((1^{\vee}_{\rho})_{\frac{l}{2}})(c)
=\frac{L(1+\frac{l}{2},\pi,\rho)}{L(-\frac{l}{2},\pi,\rho^{\vee})}\CS((1^{\vee}_{\rho})_{\frac{l}{2}})(c).
\end{align*}
Here we notice that if $\vphi_{\pi}$ is the zonal spherical function of $\pi$, then $\vphi(g^{-1})$ is exactly the zonal spherical function of $\pi^{\vee}$, so we get $\CS(1_{\rho}^{\vee})(c) = L(0,\pi_{c}^{\vee},\rho)$.
Hence
\begin{align*}
\CS((1^{\vee}_{\rho})_{\frac{l}{2}})(c) = L(-\frac{l}{2},\pi^{\vee},\rho) = L(-\frac{l}{2},\pi,\rho^{\vee}).
\end{align*}
Therefore
\begin{align*}
\CS(|\sig|^{l+1}\CF_{\rho}(1_{\rho,-\frac{l}{2}}))&=
\CS(\Phi^{K}_{\psi,\rho}*(1^{\vee}_{\rho})_{\frac{l}{2}})
=\frac{L(1+\frac{l}{2},\pi,\rho)}{L(-\frac{l}{2},\pi,\rho^{\vee})}L(-\frac{l}{2},\pi,\rho^{\vee})
\\&=
L(1+\frac{l}{2},\pi,\rho) = \CS(1_{\rho,1+\frac{l}{2}})=\CS(1_{\rho,-\frac{l}{2}} |\sig|^{l+1}).
\end{align*}
Using the inverse Satake isomorphism, it follows that $\CF_{\rho}(1_{\rho,-\frac{l}{2}}) = 1_{\rho,-\frac{l}{2}}$.
\end{proof}

Finally we are going to verify the spherical part of Conjecture \ref{operator}.
\begin{pro}
$\CF_{\rho}$ preserves the space $\CS_{\rho}(G,K)$.
\end{pro}
\begin{proof}
To show that $\CF_{\rho}$ preserves the space $\CS_{\rho}(G,K)$,
we only need to show that for any $f\in \CH(G,K)$, as a rational function on $\wh{T}/W$
\begin{align*}
\frac{\CS(\CF_{\rho}(1_{\rho,-\frac{l}{2}}*f))}{L(-\frac{l}{2},\pi_{c},\rho)}
\end{align*}
lies in $\BC[\wh{T}/W]$.

By definition,
\begin{align*}
\CF_{\rho}(1_{\rho,-\frac{l}{2}}*f) = |\sig|^{-l-1}(\Phi_{\psi,\rho}^{K}*((1_{\rho,-\frac{l}{2}}*f)^{\vee}))
=|\sig|^{-l-1} \Phi_{\psi,\rho}^{K}*f^{\vee}*1_{\rho,-\frac{l}{2}}^{\vee}.
\end{align*}
Since $\CH(G,K)$ is commutative, and functions in $\CH(G,K)$ also commute with $1_{\rho,s}$, we get
\begin{align*}
 \Phi_{\psi,\rho}^{K}*f^{\vee}*1_{\rho,-\frac{l}{2}}^{\vee} = \Phi_{\psi,\rho}^{K}*1_{\rho,-\frac{l}{2}}^{\vee}*f^{\vee}.
\end{align*}
As shown in the proof of Proposition \ref{fixbasic}, we know that $\Phi_{\psi,\rho}^{K}*1^{\vee}_{\rho,-\frac{l}{2}} = 1_{\rho,1+\frac{l}{2}}$. Therefore we only need to show
\begin{align*}
|\sig|^{-l-1}(1_{\rho,1+\frac{l}{2}}*f^{\vee})\in \CS_{\rho}(G,K),
\end{align*}
which, after applying the Satake isomorphism, is equivalent to showing that
\begin{align*}
\CS(1_{\rho,1+\frac{l}{2}}*f^{\vee}) \subset L(1+\frac{l}{2},\pi_{c},\rho)\BC[\wh{T}/W].
\end{align*}
But this follows from the definition.
\end{proof}

\begin{pro}\label{unitaryprop}
$\CF_{\rho}$ extends to a unitary operator on the space $L^{2}(G,K,|\sig|^{l+1}dg)$.
\end{pro}
\begin{proof}
To show that that $\CF_{\rho}$ extends to a unitary operator on the space $L^{2}(G,K,|\sig|^{l+1}dg)$, we only need to show the equality
\begin{align*}
<\CF_{\rho}(f),\CF_{\rho}(h)>_{L^{2}(G,K,|\sig|^{l+1}dg)} = <f,h>_{L^{2}(G,K,|\sig|^{l+1}dg)}
\end{align*}
for all $f$ and $h$ in $\CH(G,K)$.

Now
\begin{align*}
<\CF_{\rho}(f),\CF_{\rho}(h)>_{L^{2}(G,K,|\sig|^{l+1}dg)} &= \int_{G}\CF_{\rho}(f)(g)\wb{\CF_{\rho}(h)}(g)|\sig (g)|^{l+1}dg
\\&= \CF_{\rho,l+1}(f)*\wb{\CF_{\rho}(h)}^{\vee}(e)
\\
<f,h>_{L^{2}(G,K,|\sig|^{l+1}dg)}
&= \int_{G}f(g)\wb{h}(g)|\sig(g)|^{l+1}dg
\\&=\wb{h}_{l+1}*f^{\vee}(e).
\end{align*}

To show that they are equal to each other, using the Satake isomorphism, it is enough to show that
as a rational function in $c\in \wh{T}/W$, we have
\begin{align*}
\CS(\CF_{\rho,l+1}(f)*\wb{\CF_{\rho}(h)}^{\vee})(c) = \CS(\wb{h}_{l+1}*f^{\vee})(c).
\end{align*}
Using the fact that $\CS$ is an algebra homomorphism, we get
\begin{align}\label{simplifytrace}
\CS(\CF_{\rho,l+1}(f)*\wb{\CF_{\rho}(h)}^{\vee}) = \CS(\CF_{\rho,l+1}(f))\CS(\wb{\CF_{\rho}(h)}^{\vee}).
\end{align}
Now
\begin{align*}
\CF_{\rho,l+1}(f)(g) &= \CF_{\rho}(f)(g)|\sig(g)|^{l+1}
\\&= \Phi_{\psi,\rho}^{K}*f^{\vee}(g),
\\
\wb{\CF_{\rho}(h)}^{\vee}(g) &= \wb{\CF_{\rho}(h)}(g^{-1}) =|\sig(g)|^{l+1}\wb{\Phi_{\psi,\rho}^{K}}*\wb{h}^{\vee}(g^{-1})
\\ &= |\sig(g)|^{l+1} (\wb{\Phi_{\psi,\rho}^{K}}*\wb{h}^{\vee})^{\vee}(g)
\\&=  |\sig(g)|^{l+1} (\wb{h} *\wb{\Phi_{\psi,\rho}^{K}}^{\vee})(g).
\end{align*}
Plug the calculations into the equation (\ref{simplifytrace}), we get that as a rational function in $c\in \wh{T}/W$, the left hand side of the equation (\ref{simplifytrace}) can be written as
\begin{align}\label{afterfourier}
\CS(\Phi^{K}_{\psi,\rho})(c)\CS(f^{\vee})(c) \CS(\wb{h})(c\cdot q^{-(l+1)})
\CS(\wb{\Phi_{\psi,\rho}^{K}}^{\vee})(c\cdot q^{-(l+1)}).
\end{align}
Similarly, the right hand side of the equation (\ref{simplifytrace}) can be written as
\begin{align}\label{beforefourier}
\CS(\wb{h}_{l+1})(c)\CS(f^{\vee})(c)
=\CS(\wb{h})(c\cdot q^{-(l+1)})\CS(f^{\vee})(c).
\end{align}
Comparing equations (\ref{afterfourier}) and (\ref{beforefourier}), we only need to show the following equality
\begin{align*}
\CS(\Phi^{K}_{\psi,\rho})(c)\CS(\wb{\Phi_{\psi,\rho}^{K}}^{\vee})(c\cdot q^{-(l+1)}) =1.
\end{align*}
First we simplify the term
\begin{align*}
\CS(\wb{\Phi_{\psi,\rho}^{K}}^{\vee})(c\cdot q^{-(l+1)}) = \CS((\wb{\Phi_{\psi,\rho}^{K}}^{\vee})_{l+1})(c).
\end{align*}
Then using the definition of $\Phi^{K}_{\psi,\rho,s}$, we have
\begin{align*}
\CS(\Phi^{K}_{\psi,\rho,s})(c) = \gam(-s-\frac{l}{2},(\pi_{c})^{\vee},\rho,\psi)
=\frac{L(1+s+\frac{l}{2},\pi_{c},\rho)}{L(-s-\frac{l}{2},\pi_{c},\rho^{\vee})}.
\end{align*}
Letting $s=0$, we get that as a rational function in $c\in \wh{T}/W$,
\begin{align*}
\CS(\Phi^{K}_{\psi,\rho})(c) = \frac{L(1+\frac{l}{2},\pi_{c},\rho)}{L(-\frac{l}{2},\pi_{c},\rho^{\vee})}
\end{align*}
By Remark \ref{realvalued}, the function $\Phi^{K}_{\psi,\rho}$ is real-valued, which means that $\wb{\Phi_{\psi,\rho}^{K}} = \Phi_{\psi,\rho}^{K}$, therefore
\begin{align*}
\CS((\wb{\Phi_{\psi,\rho}^{K}}^{\vee})_{l+1})(c) &=
\CS(\wb{\Phi_{\psi,\rho,-(l+1)}^{K}})(c^{-1})
=\CS(\Phi_{\psi,\rho,-(l+1)}^{K})(c^{-1})
\\&=
\frac{L(-\frac{l}{2},\pi_{c}^{\vee},\rho)}{L(\frac{l}{2}+1,\pi_{c}^{\vee},\rho^{\vee})}
= \frac{L(-\frac{l}{2},\pi_{c},\rho^{\vee})}{L(\frac{l}{2}+1,\pi_{c},\rho)} = \CS(\Phi^{K}_{\psi,\rho,s})(c)^{-1}.
\end{align*}
It follows that
\begin{align*}
\CS(\CF_{\rho,l+1}(f)*\wb{\CF_{\rho}(h)}^{\vee})(c) = \CS(\wb{h}_{l+1}*f^{\vee})(c)
\end{align*}
as a rational function in $c\in \wh{T}/W$.
Using the inverse Satake isomorphism we get the desired equality
\begin{align*}
\CF_{\rho,l+1}(f)*\wb{\CF_{\rho}(h)}^{\vee} = \wb{h}_{l+1}*f^{\vee}.
\end{align*}
\end{proof}

\section{Archimedean Case}

In this section we study asymptotic properties for $1_{\rho,s}$ and $\Phi^{K}_{\psi,\rho,s}$ when $F$ is an archimedean field.

First we give the definition of $1_{\rho,s}$ and $\Phi^{K}_{\psi,\rho,s}$.
\begin{defin}
The basic function $1_{\rho,s}$ is defined to be the smooth bi-$K$-invariant function on $G$ such that
\begin{align*}
\int_{G} 1_{\rho,s} (g)\vphi_{\pi}(g)dg = L(s,\pi,\rho),
\end{align*}
where $\vphi_{\pi}$ is the zonal spherical function associated to the spherical representation $\pi$ of $G$, and $1_{\rho,s}  = 1_{\rho}|\sig|^{s}$.
\end{defin}

\begin{defin}
The spherical component of the distribution kernel of $\rho$-Fourier transform kernel $\Phi_{\psi,\rho,s}$, which we denote by $\Phi^{K}_{\psi,\rho,s}$, is defined to be the smooth bi-$K$-invariant function on $G$ such that
\begin{align*}
\int_{G} \Phi_{\psi,\rho,s}^{K} (g)\vphi_{\pi}(g)dg = \gam(-s-\frac{l}{2},\pi^{\vee},\rho,\psi),
\end{align*}
where $\vphi_{\pi}$ is the zonal spherical function associated to the spherical representation $\pi$ of $G$,
and $\Phi_{\psi,\rho,s}^{K} = \Phi_{\psi,\rho}^{K}|\sig|^{s}$.
\end{defin}

By the \emph{spherical Plancherel transform}, we know that the analytical properties of $1_{\rho,s}$ and $\Phi_{\psi,\rho,s}^{K}$ are completely determined by the corresponding analytical properties of $L(s,\pi,\rho)$ and $\gam(s,\pi,\rho,\psi)$.

\subsection{Spherical Plancherel Transform}\label{sec:spplancherel}
In this section, we review the theory of spherical plancherel transform for any real reductive Lie group belonging to the \emph{Harish-Chandra class} as defined in \cite[Definition~2.1.1]{gangolli2012harmonic}. In particular, it applies to our situation. The main references are \cite{anker1991spherical} and \cite{gangolli2012harmonic}.

Let $\Fg$ be the Lie algebra of $G$. We fix the Cartan decomposition $\Fg = \Fk \oplus \Fp$, where $\Fk$ is the Lie algebra of $K$. For any $\lam\in \Fa^{*}$, where $\Fa$ is the maximal abelian subalgebra of $\Fp$,  we let $\pi_{\lam}$ be the spherical representation induced from the character
\begin{align*}
m \exp(H) n \to e^{i\lam(H)}, H\in \Fa
\end{align*}
of the minimal parabolic subgroup $P=MAN$. Here $A=\exp \Fa$, $M$ is the centralizer of $A$ in $K$, and $N$ is the corresponding unipotent radical.

We denote the zonal spherical function of  $\pi_{\lam}$ by $\vphi_{\lam}$.

We fix the norm $|\cdot|$ induced by the Killing form on $G$ as in \cite[1~Preliminaries]{anker1991spherical}.

The elements of $U(\Fg)$ acts on $C^{\infty}(G)$ as differential operators. Following \cite[1~Preliminaries]{anker1991spherical}, for any $(D,E)\in U(\Fg)\times U(\Fg)$ and $f\in C^{\infty}(G)$, $x\in G$, we can define the left $D$ right $E$ derivative $f(D;x;E)$ of $f$, which again lies in $C^{\infty}(G)$.

We introduce the function spaces $S^{p}(K\bs G/K)$ and $S(\Fa^{*}_{\veps})$. Here $0<p\leq 2$ is any real number, and
$\veps = \frac{2}{p}-1$.

\begin{defin}
For $0<p\leq 2$, let $S^{p}(K\bs G/K)$ be the space of bi-$K$-invariant functions $f$ in $C^{\infty}(K\bs G/K)$ such that the following norm
\begin{align*}
\sig^{(p)}_{D,E,s}(f) = \sup_{x\in G} (|x|+1)^{s}\vphi_{0}(x)^{-\frac{2}{p}} |f(D;x;E)|
\end{align*}
is finite for any $D,E\in U(\Fg)$, $s\in \BZ^{+}$.
\end{defin}
Using the natural convolution structure of two bi-$K$-invariant functions, we can prove that $S^{p}(K\bs G/K)$ is a Frech\'et algebra, where the topology is induced by the semi-norms given by $\{ \sig^{(p)}_{D,E,s} | \quad D,E\in U(\Fg), s\in \BZ^{+}  \}$. Moreover, as mentioned in \cite[Lemma~6]{anker1991spherical}, the space $C^{\infty}_{c}(K\bs G/K)$ is a dense subspace of $S^{p}(K\bs G/K)$.

Now we introduce the space $S(\Fa^{*}_{\veps})$.
\begin{defin}
Let $C^{\veps \rho}$ be the convex hull generated by $W\cdot \veps \rho_{B}$ in $\Fa^{*}$.
Let $\Fa^{*}_{\veps} = \Fa^{*}+iC^{\veps \rho_{B}}$. Then $S(\Fa^{*}_{\veps})$ consists of complex valued functions $h$ on $\Fa^{*}_{\veps}$ such that the following holds.
\begin{enumerate}
\item $h$ is holomorphic in the interior of $\Fa^{*}_{\veps}$.

\item $h$ and all its derivatives extend continuously to $\Fa^{*}_{\veps}$.

\item For any polynomial function $P$ on $\Fa^{*}_{\veps}$, $t\in \BZ^{+}$,
\begin{align*}
\tau^{(\veps)}_{P,t} (h) = \sup_{\lam\in \Fa^{*}_{\eps}} (|\lam|+1)^{t} |P(\frac{\partial}{\partial \lam})h(\lam)|
\end{align*}
is finite.
\end{enumerate}
\end{defin}

Let $S(\Fa^{*}_{\veps})^{W}$ be the $W$-invariant elements in $S(\Fa^{*}_{\veps})$.
We can show that $S(\Fa^{*}_{\veps})^{W}$ is a Frech\'et algebra, where the algebra structure is given by pointwise multiplication, and the $W$-invariant Paley-Wiener functions on $\Fa^{*}_{\BC}$, denoted by $\CP(\Fa^{*}_{\BC})^{W}$, is a dense subspace of $S(\Fa^{*}_{\veps})^{W}$ after restricted to $\Fa^{*}_{\veps}$.

In particular, when $\veps=0$, $S(\Fa^{*})$  is the classical Schwartz space on $\Fa^{*}$.

\begin{defin}
For any $f\in S^{p}(K\bs G/K)$, $\lam\in \Fa^{*}$, let $\CH$ be the spherical transform defined by
\begin{align*}
\CH(f) (\lam) = \int_{G} f(x)\vphi_{\lam}(x)dx.
\end{align*}
\end{defin}

\begin{thm}[\cite{anker1991spherical},\cite{gangolli2012harmonic}]\label{isothm}
\begin{enumerate}
\item $\CH$ is a topological isomorphism of Frech\'et algebra between $S^{p}(K\bs G/K)$ and $S(\Fa^{*}_{\veps})^{W}$, where $0<p\leq 2$ and $\veps = \frac{2}{p}-1$.

\item The inverse transform is given by
$$
\CH^{-1}(h)(x) = \textrm{const}\int_{\Fa^{*}}d\lam |c(\lam)|^{-2} h(-\lam)\vphi_{\lam}(x).
$$

\end{enumerate}
\end{thm}

\subsection{Langlands Classification for $\GL_{n}(\BR)$: Spherical Case}\label{sec:corrreal}
Before coming to study the analytical properties of $L$-functions and $\gam$-factors, we need to obtain an explicit formula for $L$-functions and $\gam$-factors. Therefore we review the Langlands classifications and Langlands correspondence of spherical representations for $\GL_{n}(\BR)$ and $\GL_{n}(\BC)$. The main reference for this and next sections is \cite{knapp55local}. For more advanced reference, the reader can consult \cite{langlands1989irreducible}.

The Langlands classification for $\GL_{n}(\BR)$ describes all irreducible admissible representations of $\GL_{n}(\BR)$ up to infinitesimal equivalence. Since we only care about the spherical representations, we only present the classification and correspondence for spherical representations of $\GL_{n}(\BR)$.

The building blocks for spherical representations of $\GL_{n}(\BR)$ are the quasi-character $a\to |a|^{t}_{\BR}$ of $\GL_{1}(\BR)$. Here $|\cdot|_{\BR}$ denotes the ordinary valuation on $\BR$, and $t\in \BC$.

We have the diagonal torus subgroup
\begin{align*}
T=\GL_{1}(\BR)\times ... \times \GL_{1}(\BR)\cong (\GL_{1}(\BR))^{n}.
\end{align*}
For each $j$ with $1\leq j\leq n$, let $\sig_{j}$ be a quasi-character of $\GL_{1}(\BR)$ of the form $a\to |a|^{t_{j}}_{\BR}$. Then by tensor product, $(\sig_{1},...,\sig_{n})$ defines a representation of the diagonal torus $T$, and we extend the representation to the corresponding Borel subgroup $B=TN$,  where $N$ is the unipotent radical. We set
\begin{align*}
I(\sig_{1},...,\sig_{n}) = \textrm{Ind}^{G}_{B}(\sig_{1},...,\sig_{n})
\end{align*}
using unitary induction.

\begin{thm}\cite[Theorem~1]{knapp55local}\label{llcreal}
For $G=\GL_{n}(\BR)$,
\begin{enumerate}
\item if the parameters $t_{j}$ of $(\sig_{1},...,\sig_{n})$ satisfy
\begin{align*}
\Re \, t_{1}\geq \Re \, t_{2} \geq ...\geq \Re \, t_{n},
\end{align*}
then $I(\sig_{1},...,\sig_{n})$ has a unique irreducible quotient $J(\sig_{1},...,\sig_{n})$.

\item the representations $J(\sig_{1},...,\sig_{n})$ exhaust the spherical representation of $G$ up to infinitesimal equivalence.

\item Two such representations $J(\sig_{1},...,\sig_{n})$ and $J(\sig^{\p}_{1},...,\sig^{\p}_{n^{\p}})$ are infinitesimally equivalent if and only if $n^{\p} = n$ and there exists a permutation $j(i)$ of $\{ 1,...,n \}$ such that $\sig^{\p}_{i} = \sig_{j(i)}$ for $1\leq i\leq n$.
\end{enumerate}
\end{thm}

Next we determine the corresponding Langlands parameters of spherical representations, which are given by homomorphisms of the abelianlization of the Weil group, which we denoted by $W^{ab}_{\BR}\cong \BC^{\times}$ into $\GL_{n}(\BR)$. Following \cite[Section~3]{knapp55local}, the Langlands parameters corresponding to spherical representations of $\GL_{n}(\BR)$ are given by the direct sum of $n$ one-dimensional representations of $\BC^{\times}$ of the following form:
\begin{align*}
(+,t): \quad \vphi(z) = |z|^{t}_{\BR}, \quad \vphi(j) = +1.
\end{align*}

Now let $\vphi$ be an $n$-dimensional semisimple complex representation of $W_{\BR}$, which is $n$ direct sum of quasi-characters of the form $(+,t)$. For any $1\leq j\leq n$, let $\vphi_{j}$ be the corresponding irreducible constituent of $\vphi$. To $\vphi_{j}$ we associate a quasi-character.
In this way, we associate a tuple $(\sig_{1},...,\sig_{n})$ of representations to $\vphi$. By permutations if necessary, the complex numbers $t_{1},...,t_{n}$ satisfy the assumption of Theorem \ref{llcreal}. Then by Theorem \ref{llcreal}, we can then make the association
\begin{align}\label{llcasso}
\vphi \to \rho_{\BR}(\vphi) = J(\sig_{1},...,\sig_{n})
\end{align}
and come to the following conclusion.

\begin{thm}\cite[Theorem~2]{knapp55local}
The association (\ref{llcasso}) is a well-defined bijection between the set of all equivalence classes of $n$-dimensional semisimple complex representations of $W_{\BR}$ which are $n$ direct sum of one-dimension representations of the form $(+,t)$, and the set of all equivalence classes of spherical representations of $\GL_{n}(\BR)$.
\end{thm}

If $\vphi$ is one-dimensional given by $(+,t)$, the associated $L$-function and $\veps$-factor are given as follows.
\begin{align*}
L(s,\vphi) &=  \pi^{-\frac{(s+t)}{2}}\Gam\Big(\frac{s+t}{2}\Big), \\
\veps(s,\pi,\psi) &= 1.
\end{align*}

For $\vphi$ reducible, $L(s,\vphi)$ and $\veps(s,\vphi,\psi)$ are the product of the $L$-functions and $\veps(s,\vphi,\psi)$ of the one-dimensional factors of $\vphi$.

\subsection{Langlands Classification for $\GL_{n}(\BC)$: Spherical Case}\label{sec:corrcplx}
The Langlands classification for $\GL_{n}(\BC)$ describes all irreducible admissible representations of $\GL_{n}(\BC)$ up to infinitesimal equivalence. Since we only care about the spherical representations, we only present the classification and correspondence for spherical representations.

The building blocks for spherical representations of the group $\GL_{n}(\BC)$ are the quasi-character $a\to |a|^{t}_{\BC}$ of $\GL_{1}(\BC)$. Here $|\cdot|_{\BC}$ denotes the ordinary valuation on $\BC$ given by
$$
|z|_{\BC} = |z\wb{z}| = |z|^{2}, \quad z\in \BC,
$$
and $t\in \BC$.

We have the diagonal torus subgroup
\begin{align*}
T=\GL_{1}(\BC)\times ... \times \GL_{1}(\BC)\cong (\GL_{1}(\BC))^{n}.
\end{align*}
For each $j$ with $1\leq j\leq n$, let $\sig_{j}$ be a quasi-character of $\GL_{1}(\BC)$ of the form $a\to |a|^{t_{j}}_{\BC}$. Then by tensor product, $(\sig_{1},...,\sig_{n})$ defines a representation of the diagonal torus $T$, and we extend the representation to the corresponding Borel subgroup $B=TN$,  where $N$ is the unipotent radical. We then set
\begin{align*}
I(\sig_{1},...,\sig_{n}) = \textrm{Ind}^{G}_{B}(\sig_{1},...,\sig_{n})
\end{align*}
using unitary induction.

\begin{thm}\cite[Theorem~4]{knapp55local}\label{llccplx}
For $G=\GL_{n}(\BC)$,
\begin{enumerate}
\item if the parameters $t_{j}$ of $(\sig_{1},...,\sig_{n})$ satisfy
\begin{align*}
\Re \, t_{1}\geq \Re \, t_{2} \geq ...\geq \Re \, t_{n},
\end{align*}
then $I(\sig_{1},...,\sig_{n})$ has a unique irreducible quotient $J(\sig_{1},...,\sig_{n})$.

\item the representations $J(\sig_{1},...,\sig_{n})$ exhaust the spherical representations of $G$ up to infinitesimal equivalence.

\item Two such representations $J(\sig_{1},...,\sig_{n})$ and $J(\sig^{\p}_{1},...,\sig^{\p}_{n^{\p}})$ are infinitesimally equivalent if and only if $n^{\p} = n$ and there exists a permutation $j(i)$ of $\{ 1,...,n \}$ such that $\sig^{\p}_{i} = \sig_{j(i)}$ for $1\leq i\leq n$.
\end{enumerate}
\end{thm}

Next we determine the corresponding Langlands parameters of spherical representations, which are given by homomorphisms of the Weil group $W_{\BC}\cong \BC^{\times}$ into $\GL_{n}(\BC)$. Following \cite[Section~4]{knapp55local}, the Langlands parameters corresponding to spherical representations of $\GL_{n}(\BC)$ are given by the direct sum of $n$ one-dimensional representations of $\BC^{\times}$ of the following form:
$$
(0,t): z\in \BC^{\times}\to |z|^{t}_{\BC}, \quad l\in \BZ, t\in \BC.
$$

Now let $\vphi$ be an $n$-dimensional semisimple complex representation of $W_{\BC}$, which is $n$ direct sum of quasi-characters of the form $(0,t)$. To $\vphi_{j}$ we associate a quasi-character $\sig_{j} = |\cdot|^{t_{j}}_{\BC}$ of $\GL_{1}(\BC)$. In this way, we associate a tuple $(\sig_{1},...,\sig_{n})$ of representations to $\vphi$. By permutations if necessary, the complex numbers $t_{1},...,t_{n}$ satisfy the assumption of Theorem \ref{llccplx}. Then by Theorem \ref{llccplx}, we can then make the association
\begin{align}\label{llcassocplx}
\vphi \to \rho_{\BC}(\vphi) = J(\sig_{1},...,\sig_{n})
\end{align}
and come to the following conclusion.

\begin{thm}\cite[Theorem~5]{knapp55local}
The association (\ref{llcassocplx}) is a well-defined bijection between the set of all equivalence classes of $n$-dimensional semisimple complex representations of $W_{\BC}$ which are $n$ direct sum of $1$-dimensions of form $(0,t)$, and the set of all equivalence classes of spherical representations of $\GL_{n}(\BC)$.
\end{thm}

If $\vphi$ is given by $(0,t)$, the associated $L$-function and $\veps$-factor are given as follows
\begin{align*}
L(s,\vphi) &=  2(2\pi)^{-(s+t)}\Gam(s+t), \\
\veps(s,\pi,\psi) &= 1.
\end{align*}

For $\vphi$ reducible, $L(s,\vphi)$ and $\veps(s,\vphi,\psi)$ are the product of the $L$-functions and $\veps(s,\vphi,\psi)$ of the irreducible constituents of $\vphi$.

\subsection{Asymptotic of $1_{\rho,s}$ and $\Phi_{\psi,\rho,s}^{K}$ : Real Case}\label{archiproof}
Based on the local Langlands correspondence, we know that in order to study the asymptotic of $L$-functions, we need to study the asymptotic of $\Gam$ function, where
\begin{align*}
\Gam(z) = \int_{0}^{\infty}x^{z-1}e^{-x}dx.
\end{align*}

Here we recall the following estimation from \cite[1.18(6)]{hightransfun}, which can easily be derived from the classical Stirling formula.
\begin{thm}\label{estimateforgamma}
For fixed $x\in \BR$,
\begin{align*}
\Gam(x+iy) = \sqrt{2\pi}|y|^{x-\frac{1}{2}} e^{-x-\frac{|y|\pi}{2}}[1+O(\frac{1}{|y|})], \quad |y|\to \infty.
\end{align*}
\end{thm}

Then we give a proof for the following estimation for the derivatives of $\Gam$-function.
\begin{thm}\label{derivativeofgamma}
We have
\begin{align*}
\lim_{|z|\to \infty, |\arg z|<\pi} \frac{\Gam^{(n)}(z)}{\Gam(z)(\log z)^{n}} =1,
\end{align*}
where $\Gam^{(n)}(z)$ is the $n$-th derivative of $\Gam(z)$.
\end{thm}
\begin{proof}
We prove the theorem via induction.

Let $D_{n}(z) = \frac{\Gam^{(n)}(z)}{\Gam(z)}$. When $n=1$, using the classical Stirling formula, we have
\begin{align*}
\log \Gam(z) = \frac{1}{2}(\log(2\pi) - \log(z)) +z (\log z - 1) + O(\frac{1}{z})
\end{align*}
for any $|z|\to \infty, |\arg z|<\pi$.
Diving both sides by $z\log z$, we get
\begin{align*}
\lim_{|z|\to \infty, |\arg z|<\pi} \frac{\log \Gam(z)}{z\log z} = 1.
\end{align*}
By the L'H\^osptial's rule, we get
\begin{align*}
\lim_{|z|\to \infty, |\arg z|<\pi} \frac{D_{1}(z)}{ 1+\log z } =1.
\end{align*}
Hence we obtain
\begin{align*}
\lim_{|z|\to \infty, |\arg z|<\pi} \frac{D_{1}(z)}{\log z } =1.
\end{align*}
Therefore we complete the proof for $n=1$.

By definition, $\Gam(z)D_{n}(z) = \Gam^{(n)}(z)$. Taking derivative on both sides, we get
\begin{align*}
\Gam^{(1)}(z)D_{n}(z) +D^{\p}_{n}(z)\Gam(z) = \Gam^{(n+1)}(z).
\end{align*}
From this we can deduce the equality
\begin{align*}
D_{n+1}(z)  = D^{\p}_{n}(z)+D_{n}(z)D_{1}(z).
\end{align*}
Hence
\begin{align*}
\frac{D_{n+1}(z)}{(\log z)^{n+1}} = \frac{D^{\p}_{n}(z)}{(\log z)^{n+1}} + \frac{D_{n}(z)D_{1}(z)}{(\log z)^{n}(\log z)}.
\end{align*}

We assume that the limit
\begin{align*}
\lim_{|z|\to \infty, |\arg z|<\pi}\frac{D_{k}(z)}{(\log z)^{k}}=1,\quad 1\leq k\leq n
\end{align*}
holds.
To show that the limit formula holds for $k=n+1$, we only need to show that
\begin{align*}
\lim_{|z|\to \infty, |\arg z|<\pi}\frac{D^{\p}_{n}(z)}{ (\log z)^{n+1}} = 0.
\end{align*}
Now we have the formula
\begin{align*}
\lim_{|z|\to \infty, |\arg z|<\pi}\frac{D_{n}(z)}{(\log z)^{n}}=1.
\end{align*}
By the L'H\^ospital's rule, we get
\begin{align*}
\lim_{|z|\to \infty, |\arg z|<\pi} \frac{zD^{\p}_{n}(z)}{n(\log z)^{n-1}}= 1.
\end{align*}
Hence
\begin{align*}
\lim_{|z|\to \infty, |\arg z|<\pi} \frac{z(\log z)^{2}D^{\p}_{n}(z)}{n(\log z)^{n+1}}= 1,
\end{align*}
and we get
\begin{align*}
\lim_{|z|\to \infty, |\arg z|<\pi}\frac{D^{\p}_{n}(z)}{ (\log z)^{n+1}} = 0.
\end{align*}
Combining the above results we prove the theorem.
\end{proof}

Then we come to describe an explicit formula for $L(s,\pi_{\lam},\rho)$.

By definition, $\pi_{\lam}$ is induced from the character
\begin{align*}
m\exp(H)n \to e^{i\lam(H)}.
\end{align*}
If we assume that $\lam = (\lam_{1},...,\lam_{m})\in \Fa^{*}$, where $m$ is $1$ plus the semisimple rank of $G$, then its associated Langlands parameter is of the form
\begin{align*}
t\in W^{ab}_{\BR}\cong \BR^{\times} \to   \left( \begin{matrix} 
      |t|^{i\lam_{1}} &  			   &	&\\
       			   & |t|^{i\lam_{2}} &	&\\
       			   &			   &	&\\
       			   &			   &\ddots&\\
			   &			   &	&|t|^{i\lam_{m}}  \\
   \end{matrix} \right).
\end{align*}

Assume that $\rho$ has weights $\vpi_{1}, \vpi_{2},...,\vpi_{n}$, where $n=\dim (V_{\rho})$. Then the associated parameter for $\rho(\pi_{\lam})$, which is the functorial lifting image of $\pi_{\lam}$ along $\rho$, is
\begin{align*}
t\to   \left( \begin{matrix} 
      |t|^{i\vpi_{1}(\lam)} &  			   &	&\\
       			   & |t|^{i\vpi_{2}(\lam)} &	&\\
       			   &			   &	&\\
       			   &			   &\ddots&\\
			   &			   &	&|t|^{i\vpi_{n}(\lam)}  \\
   \end{matrix} \right),
\end{align*}
where $\vpi_{j}(\lam) = \sum_{k=1}^{m}n^{j}_{k}\lam_{k}$, $n_{k}\in \BZ_{\geq 0}$.

In the following, we need to use the following lemma on the representation $\rho$:
\begin{lem}\label{assumpweight}
We have the following inequality
\begin{align}\label{inequalityweight}
\sum_{k=1}^{n}|\vpi_{k}(x)| \geq C_{\rho}\sum_{t=1}^{m}|x_{t}|,\quad \textrm{for all } x=(x_{1},...,x_{m})\in \Fa^{*}
\end{align}
for some constant $C_{\rho}>0$.
\end{lem}
\begin{proof}
The basic ingredient that we use is the fact that the representation $\rho$ is faithful.

We restrict $\rho$ to the split torus $(\BC^{\times})^{m}$ of $\LG$. Up to conjugation, we can view $\rho$ as an injective homorphism from $(\BC^{\times})^{m}$ to $(\BC^{\times})^{n}$, where $n=\dim (V_{\rho})$. Passing to Lie algebra, we get an injective homomorphism from $(\BC)^{m}$ to $(\BC)^{n}$, which is given by the direct sum of $\vpi_{k}$, $1\leq k\leq n$. Here each $\vpi_{k}$ can be viewed as a character of $(\BC)^{m}$.

We notice that the inequality (\ref{inequalityweight}) is invariant by scaling, and holds identically when $x = (x_{1},...,x_{m}) = 0$. Therefore in order to obtain the bound $C_{\rho}$, we can assume that 
$\sum_{t=1}^{m}|x_{t}| = 1$. In this case, the following function $f(x)$
$$
f(x) = \sum_{k=1}^{n}|\vpi_{k}(x)|, \quad \sum_{t=1}^{m}|x_{t}|=1
$$
is continuous. Using the fact that the equality $ \sum_{t=1}^{m}|x_{t}|=1$ defines a compact set in $(\BC)^{m}$, we notice that there exists $x\in (\BC)^{m}$ with the property $\sum_{t=1}^{m}|x_{t}|=1$, such that $f(x)$ is maximal. We let $C_{\rho}$ to be the maximum. 

Now if $C_{\rho}$ is equal to $0$, this means that $\vpi_{k}(x) = 0$ for all $1\leq k\leq n$. In particular, it means that the morphism $\rho$ is not injective when restricted to the Lie algebra $(\BC)^{m}$, which is a contraction.

It follows that $C_{\rho}>0$. This completes the proof.
\end{proof}

Now we are going to state our result on an asymptotic of $1_{\rho,s}$.
\begin{thm}\label{asympforbasic}
If $\Re(s)$ satisfies the following inequality
\begin{align*}
\Re(s)>\textrm{max}\{\vpi_{k}(\mu) |\quad 1\leq k\leq n,  \mu \in C^{\veps \rho_{B}}\},
\end{align*}
then $1_{\rho,s}$ belongs to  $S^{p}(K\bs G/K)$. Here $\veps = \frac{2}{p}-1$, $0<p\leq 2$, and $\{\vpi_{k}\}_{k=1}^{n}$ are the weights of the representation $\rho: \LG\to \GL(V_{\rho})$.
\end{thm}
\begin{proof}
By definition
\begin{align*}
L(s,\pi_{\lam},\rho) =  \prod_{k=1}^{n}\pi^{-(\frac{s+i\vpi_{k}(\lam)}{2})}\Gam \Big( \frac{s+i\vpi_{k}(\lam)}{2}\Big).
\end{align*}

When $\Re(s)$ is sufficiently large, we want to show that the function $L(s,\pi_{\lam},\rho)$, as a function of $\lam$, lies in the space $S(\Fa^{*}_{\veps})^{W}$. The $W$-invariance of the function follows from the fact that
$\pi_{w\lam}\cong \pi_{\lam}$ for any $w\in W$. Therefore we only need to show the following semi-norm for $L(s,\pi_{\lam},\rho)$
\begin{align*}
\tau^{(\veps)}_{P,t}(L(s,\pi_{\lam},t)) = \sup_{\lam\in \Fa^{*}_{\veps}} (|\lam|+1)^{t} P(\frac{\partial}{\partial \lam})
L(s,\pi_{\lam},\rho)
\end{align*}
is finite if $\Re(s)$ is bigger than $\textrm{max}\{\vpi_{k}(\mu) | \quad 1\leq k\leq n,  \mu \in C^{\veps \rho_{B}}\}$. The reason that we need this bound is to prevent from touching the possible poles of $L(s,\pi_{\lam},\rho)$.

Now we are going to estimate
\begin{align*}
\sup_{\lam\in \Fa^{*}_{\veps}}(|\lam|+1)^{t}P(\frac{\partial}{\partial \lam})[\prod_{k=1}^{n}\pi^{-(\frac{s+i\vpi_{k}(\lam)}{2})}\Gam\Big(\frac{s+i\vpi_{k}(\lam)}{2}\Big)].
\end{align*}

The term
\begin{align*}
P(\frac{\partial}{\partial \lam})\pi^{-(\frac{s+i\vpi_{k}(\lam)}{2})}
\end{align*}
is dominated by
\begin{align*}
C_{1}(|\lam|+1)^{a}\pi^{-(\frac{s+i\vpi_{k}(\lam)}{2})}
\end{align*}
for some $a>0$ and constant $C_{1}>0$.

For the term
\begin{align*}
P(\frac{\partial}{\partial \lam})\Gam\Big(\frac{s+i\vpi_{k}(\lam)}{2}\Big),
\end{align*}
using Theorem \ref{derivativeofgamma} for the estimation on the derivative of $\Gam(z)$, it is dominated by
\begin{align*}
C_{2}(|\lam|+1)^{b}\Gam\Big(\frac{s+i\vpi_{k}(\lam)}{2}\Big)
\end{align*}
for some $b>0$ and some constant $C_{2}>0$. Here we use the fact that $\log (z)$ is dominated by $C(|z|+1)$ for some constant $C$ if $\Re(z)$ is bigger than $\textrm{max}\{\vpi_{k}(\mu) | \quad 1\leq k\leq n,  \mu \in C^{\veps \rho_{B}}\}$.

Hence we only need to show that the following term is bounded
\begin{align*}
\sup_{\lam\in \Fa^{*}_{\veps}}(|\lam|+1)^{t}\prod_{k=1}^{n}\pi^{-(\frac{s+i\vpi_{k}(\lam)}{2})}\Gam\Big(\frac{s+i\vpi_{k}(\lam)}{2}\Big).
\end{align*}

When $\lam\in \Fa^{*}_{\veps} = \Fa^{*}+iC^{\veps \rho }$, the real part of $\frac{s+i\vpi_{k}(\lam)}{2}$ is bounded and lies in a compact set, so the function $\pi^{-(\frac{s+i\vpi_{k}(\lam)}{2})}$ is always bounded.
Using Theorem \ref{estimateforgamma} for the estimation for $\Gam(x+iy)$ for $x\in \BR$ fixed, we have
\begin{align*}
&\sup_{\lam\in \Fa^{*}_{\veps}}(|\lam|+1)^{t}\prod_{k=1}^{n}\pi^{-(\frac{s+i\vpi_{k}(\lam)}{2})}\Gam\Big(\frac{s+i\vpi_{k}(\lam)}{2}\Big)\leq \\
&\sup_{\lam\in \Fa^{*}_{\veps}}C (|\lam|+1)^{t}(\sqrt{2\pi})^{n}
\prod_{k=1}^{n}[ |\frac{\Im(s)+\vpi_{k}(x)}{2}|^{\frac{\Re(s) -\vpi_{k}(y)-1}{2}}
\\
&\cdot e^{\frac{\vpi_{k}(y)}{2} -\frac{\Re(s)}{2}-\frac{|\Im(s)+\vpi_{k}(x)|\pi}{4}}]
\end{align*}
for some constant $C>0$. Here we write $\lam = x+iy$ with $x\in \Fa^{*}$, $y\in C^{\veps \rho}$.

Now we know that $s\in \BC$ is fixed, and $y$ lies in $C^{\veps\rho}$, which is a compact set. The term $\vpi_{k}(\lam)$ is also dominated by a polynomial function in $|\lam|+1$. Therefore up to a constant and a polynomial in $(|\lam|+1)$, we only need to evaluate the following term
\begin{align*}
\sup_{x\in \Fa^{*}}(|x|+1)^{t}\prod_{k=1}^{n}e^{-\frac{|\vpi_{k}(x)|\pi}{4}}.
\end{align*}
By Lemma \ref{assumpweight}, it is bounded by
\begin{align*}
\sup_{x\in \Fa^{*}}(|x|+1)^{t}\prod_{k=1}^{m}e^{-\frac{C_{\rho}|x_{k}|\pi}{4}}.
\end{align*}
which is bounded by a constant. This proves the theorem.
\end{proof}

\begin{rmk}
As mentioned in \cite{getz2015nonabelian}, by the recent work on Arthur-Selberg trace formula \cite{finis2011continuity} \cite{finis2016continuity} \cite{finis2011spectral}, the Arthur-Selberg trace formula is valid for functions in $\CS^{p}(K\bs G/K)$ whenever $0<p\leq 1$. Therefore our result gives an explicit bound of the parameter $s$ when the basic function $1_{\rho,s}$ can be plugged into the Arthur-Selberg trace formula.
\end{rmk}

We can also prove an asymptotic for $\Phi_{\psi,\rho,s}^{K}$.
By definition, the spherical component of $\Phi_{\psi,\rho,s}^{K}$ is determined via the following identity
\begin{align*}
\CH(\Phi_{\psi,\rho,s}^{K}) = \frac{L(1+s+\frac{l}{2},\pi,\rho)}{L(-s-\frac{l}{2},\pi^{\vee},\rho)}.
\end{align*}
Here we notice that if $\pi$ has Langlands parameter
\begin{align*}
t \to   \left( \begin{matrix} 
      |t|^{i\lam_{1}} &  			   &	&\\
       			   & |t|^{i\lam_{2}} &	&\\
       			   &			   &	&\\
       			   &			   &\ddots	&\\
			   &			   &	&|t|^{i\lam_{m}}  \\
   \end{matrix} \right),
\end{align*}
then $\pi^{\vee}$ has Langlands parameter
\begin{align*}
t\to   \left( \begin{matrix} 
      |t|^{-i\lam_{1}} &  			   &	&\\
       			   & |t|^{-i\lam_{2}} &	&\\
       			   &			   &	&\\
       			   &			   &\ddots	&\\
			   &			   &	&|t|^{-i\lam_{m}}  \\
   \end{matrix} \right).
\end{align*}

We first simplify the expression for $\gam$-factor by the functional equation of $\Gam(z)$.
\begin{lem}
The formula $\CH(\Phi_{\psi,\rho,s}^{K}) =
\frac{L(1+s+\frac{l}{2},\pi,\rho)}{L(-s-\frac{l}{2},\pi^{\vee},\rho)}$ can be simplified to be
\begin{align*}
\prod_{k=1}^{n} [\pi^{-(\frac{1}{2}+s+\frac{l}{2}+i\vpi_{k}(\lam)    ) }
& \Gam\Big(\frac{1+s+\frac{l}{2}+i\vpi_{k}(\lam)}{2} \Big)
 \\
 \frac{1}{\pi}\sin\Big(\pi(\frac{2+s+\frac{l}{2}+i\vpi_{k}(\lam)}{2})\Big)
&\Gam\Big(\frac{2+s+\frac{l}{2}+i\vpi_{k}(\lam)}{2}\Big)].
\end{align*}
\end{lem}
\begin{proof}
Using the definition of $L$-function, we have
\begin{align*}
\frac{L(1+s+\frac{l}{2},\pi,\rho)}{L(-s-\frac{l}{2},\pi^{\vee},\rho)}=
\frac{\prod_{k=1}^{n}\pi^{-\frac{1+s+\frac{l}{2}+i\vpi_{k}(\lam)}{2}} \Gam\big(\frac{1+s+\frac{l}{2}+i\vpi_{k}(\lam)}{2} \big) }
{\prod_{k=1}^{n}  \pi^{\frac{s+\frac{l}{2}+i\vpi_{k}(\lam)}{2}} \Gam\big(  -\frac{s+\frac{l}{2}+i\vpi_{k}(\lam)}{2}\big)       }
\\=
\prod_{k=1}^{n} \pi^{-(\frac{1}{2}+s+\frac{l}{2}+i\vpi_{k}(\lam)    ) }
\frac{ \Gam\big(\frac{1+s+\frac{l}{2}+i\vpi_{k}(\lam)}{2} \big)}{\Gam\big(  -\frac{s+\frac{l}{2}+i\vpi_{k}(\lam)}{2}\big) }.
\end{align*}
Using the functional equation for $\Gam(z)$
\begin{align*}
\Gam(z)\Gam(1-z) = \frac{\pi}{\sin(\pi z)},
\end{align*}
we obtain
\begin{align*}
\frac{1}{\Gam\big(-\frac{s+\frac{l}{2}+i\vpi_{k}(\lam)}{2}\big)} = \frac{1}{\pi}\sin\Big(\pi(\frac{2+s+\frac{l}{2}+i\vpi_{k}(\lam)}{2})\Big)
\Gam\Big(\frac{2+s+\frac{l}{2}+i\vpi_{k}(\lam)}{2}\Big).
\end{align*}
It follows that
\begin{align*}
\prod_{k=1}^{n} \pi^{-(\frac{1}{2}+s+\frac{l}{2}+i\vpi_{k}(\lam)    ) }
&\frac{ \Gam\big(\frac{1+s+\frac{l}{2}+i\vpi_{k}(\lam)}{2} \big)}{\Gam\big(  -\frac{s+\frac{l}{2}+i\vpi_{k}(\lam)}{2}\big) }
\\=
\prod_{k=1}^{n} [\pi^{-(\frac{1}{2}+s+\frac{l}{2}+i\vpi_{k}(\lam)    ) }
 &\Gam\Big(\frac{1+s+\frac{l}{2}+i\vpi_{k}(\lam)}{2}\Big)
 \\
\cdot \frac{1}{\pi}\sin\Big(\pi(\frac{2+s+\frac{l}{2}+i\vpi_{k}(\lam)}{2})\Big)
&\Gam\Big(\frac{2+s+\frac{l}{2}+i\vpi_{k}(\lam)}{2}\Big)].
\end{align*}
\end{proof}

We write $\lam= x+iy$ with $x\in \Fa^{*}$ and $y\in C^{\veps\rho_{B}}$, and we notice that the function $ \frac{1}{\pi}\sin\Big(\pi(\frac{2+s+\frac{l}{2}+i\vpi_{k}(\lam)}{2})\Big)$ is a Paley-Wiener function in $\lam$, hence lies in $\CS(\Fa^{*}_{\veps})$ as the space $\CS(\Fa^{*}_{\veps})$ contains all the Paley-Wiener functions. The function $\pi^{-(\frac{1}{2}+s+\frac{l}{2}+i\vpi_{k}(\lam)    ) }$ is bounded.
Then combining with Theorem \ref{asympforbasic} and the fact that $\CS(\Fa^{*}_{\veps})$ is a Fr\'echet algebra,
we know that if $\Re(s+1+\frac{l}{2})$ is bigger than $\textrm{max}\{\vpi_{k}(\mu) | \quad 1\leq k\leq n,  \mu \in C^{\veps \rho_{B}}\}$ and $\Re(s+2+\frac{l}{2})$ is bigger than $\textrm{max}\{\vpi_{k}(\mu) | 1\leq k\leq n,  \mu \in C^{\veps \rho_{B}}\}$, the function $\CH(\Phi_{\psi,\rho,s}^{K})$ lies in $\CS(\Fa^{*}_{\veps})$. Using the fact that $\pi_{\lam}\cong \pi_{w\lam}$ for $w\in W$, we know that $\CH(\Phi_{\psi,\rho,s}^{K})$ lies in $\CS(\Fa^{*}_{\veps})^{W}$.

In other words, we have proved the following asymptotic for $\Phi_{\psi,\rho,s}^{K}$
\begin{thm}
If $\Re(s)$ satisfies the following inequality
\begin{align*}
\Re(s)>-1-\frac{l}{2}+\textrm{max}\{\vpi_{k}(\mu) | \quad 1\leq k\leq n,  \mu \in C^{\veps \rho_{B}}\},
\end{align*}
then the function $\Phi_{\psi,\rho,s}^{K}$ lies in $\CS^{p}(K\bs G/K)$.
\end{thm}

We can also show that the Fourier transform $\CF_{\rho}$ preserves $1_{\rho,-\frac{l}{2}}$. The proof is just the same as the $p$-adic case by verifying that they have the same image under spherical Plancherel transform.

\begin{rmk}
We make a remark on the function space $\CS_{\rho}(G,K)$. In \cite{gjzeta}, the authors defined the space $\CS_{\std}(G)$ to be the derivatives of the basic function $1_{\std}$, which is not the restriction of the classical Schwartz functions on $\RM_{n}$ to $G$. Using the classical theory of Fourier transform, one can show that $\CS_{\std}(G)$ is fixed by $\CF_{\std}$. Moreover, using Casselman's subrepresentation theorem \cite{subrepresentation}, one can show that the function space $\CS_{\std}(G)$ is enough for us to obtain the standard $L$-factors.

Let $\BC[\Fg]$ be the polynomial ring on $\Fg$ and let $U(\Fg)$ be the universal enveloping algebra of $\Fg$.
Since $\CS_{\std}(G)$ is invariant under multiplication by $\BC[\Fg]$ and $U(\Fg)$, the function space $\CS_{\std}(G)$ is a Weyl algebra module, which means that the space $\CS_{\std}(G)$ has a nice algebraic structure. It seems that $\CS_{\std}(G)$ defined in \cite{gjzeta} does not carry any natural topological structure. In general, we might hope that our function space $\CS_{\rho}(G)$ carries natural topological structure like the Fr\'echet topology on classical Schwartz space.

On the other hand, one may ask why we do not set up our space $\CS_{\rho}(G,K)$ to be just $1_{\rho,-\frac{l}{2}}*C^{\infty}_{c}(G,K)$ as in $p$-adic case. Here we notice that the $L$-factor cannot be written as the fraction of two functions in the Paley-Wiener space $\CP(\Fa^{*}_{\BC})$, since the function $\Gam(z)$ satisfies the limit
\begin{align*}
\lim_{|z|\to \infty, |\arg z|<\pi}\frac{\Gam(z)}{e^{z\log z}}=1.
\end{align*}
In other words, the function space $L(-\frac{l}{2},\pi_{\lam},\rho)\CP(\Fa^{*}_{\BC})$ does not contain $\CP(\Fa^{*}_{\BC})$ as a proper subspace. We can define $\CS_{\rho}(G,K)$ to be the space of functions generated additively by $1_{\rho,-\frac{l}{2}}$, $C_{c}^{\infty}(G,K)$ and $\CF_{\rho}(C^{\infty}_{c}(G,K))$. Then $\CS_{\rho}(G,K)$ naturally contains $1_{\rho,-\frac{l}{2}}$ and is fixed by $\CF_{\rho}$, but the algebraic and topological structure is not clear as the $p$-adic case.
\end{rmk}

\subsection{Asymptotic of $1_{\rho,s}$ and $\Phi_{\psi,\rho,s}^{K}$ : Complex Case}\label{archiproofcplx}
Following the proof in the real case, we describe an explicit formula for $L(s,\pi_{\lam},\rho)$.

By definition, $\pi_{\lam}$ is induced from the character
\begin{align*}
m\exp(H)n \to e^{i\lam(H)}.
\end{align*}
If we assume that $\lam = (\lam_{1},...,\lam_{m})\in \Fa^{*}$, where $m$ is $1$ plus the semisimple rank of $G$, then its associated Langlands parameter is of the form
\begin{align*}
t\in W_{\BC}\cong \BC^{\times} \to   \left( \begin{matrix} 
      |t|^{i\lam_{1}} &  			   &	&\\
       			   & |t|^{i\lam_{2}} &	&\\
       			   &			   &	&\\
       			   &			   &\ddots&\\
			   &			   &	&|t|^{i\lam_{m}}  \\
   \end{matrix} \right)
\\=
 \left( \begin{matrix} 
      |t|_{\BC}^{\frac{i\lam_{1}}{2}} &  			   &	&\\
       			   & |t|_{\BC}^{\frac{i\lam_{2}}{2}} &	&\\
       			   &			   &	&\\
       			   &			   &\ddots&\\
			   &			   &	&|t|_{\BC}^{\frac{i\lam_{m}}{2}}  \\
   \end{matrix} \right)
   .
\end{align*}

Assume that $\rho$ has weights $\vpi_{1}, \vpi_{2},...,\vpi_{n}$, where $n=\dim (V_{\rho})$. Then the associated parameter for $\rho(\pi_{\lam})$, which is the functorial lifting image of $\pi_{\lam}$ along $\rho$, is
\begin{align*}
t\to   \left( \begin{matrix} 
      |t|_{\BC}^{\frac{i\vpi_{1}(\lam)}{2}} &  			   &	&\\
       			   & |t|_{\BC}^{\frac{i\vpi_{2}(\lam)}{2}} &	&\\
       			   &			   &	&\\
       			   &			   &\ddots&\\
			   &			   &	&|t|_{\BC}^{\frac{i\vpi_{n}(\lam)}{2}}  \\
   \end{matrix} \right),
\end{align*}
where $\vpi_{j}(\lam) = \sum_{k=1}^{m}n^{j}_{k}\lam_{k}$, $n_{k}\in \BZ_{\geq 0}$.

Now we are going to state our result on an asymptotic of $1_{\rho,s}$.
\begin{thm}\label{asympforbasiccplx}
If $\Re(s)$ satisfies the following inequality
\begin{align*}
\Re(s)>\textrm{max}\{\frac{\vpi_{k}(\mu)}{2} |\quad 1\leq k\leq n,  \mu \in C^{\veps \rho_{B}}\},
\end{align*}
then $1_{\rho,s}$ belongs to  $S^{p}(K\bs G/K)$. Here $\veps = \frac{2}{p}-1$, $0<p\leq 2$, and $\{\vpi_{k}\}_{k=1}^{n}$ are the weights of the representation $\rho: \LG\to \GL(V_{\rho})$.
\end{thm}
\begin{proof}
By definition
\begin{align*}
L(s,\pi_{\lam},\rho) =  \prod_{k=1}^{n}2(2\pi)^{-(\frac{2s+i\vpi_{k}(\lam)}{2})}\Gam\Big(\frac{2s+i\vpi_{k}(\lam)}{2}\Big).
\end{align*}

When $\Re(s)$ is sufficiently large, we want to show that the function $L(s,\pi_{\lam},\rho)$, as a function of $\lam$, lies in the space $S(\Fa^{*}_{\veps})^{W}$. The $W$-invariance of the function follows from the fact that
$\pi_{w\lam}\cong \pi_{\lam}$ for any $w\in W$. Therefore we only need to show the following semi-norm for $L(s,\pi_{\lam},\rho)$
\begin{align*}
\tau^{(\veps)}_{P,t}(L(s,\pi_{\lam},t)) = \sup_{\lam\in \Fa^{*}_{\veps}} (|\lam|+1)^{t} P(\frac{\partial}{\partial \lam})
L(s,\pi_{\lam},\rho)
\end{align*}
is finite if $\Re(s)$ is bigger than $\textrm{max}\{\frac{\vpi_{k}(\mu)}{2} | \quad 1\leq k\leq n,  \mu \in C^{\veps \rho_{B}}\}$.

Now we are going to estimate
\begin{align*}
\sup_{\lam\in \Fa^{*}_{\veps}}(|\lam|+1)^{t}P(\frac{\partial}{\partial \lam})[\prod_{k=1}^{n}
(2\pi)^{-(\frac{2s+i\vpi_{k}(\lam)}{2})}\Gam\Big(\frac{2s+i\vpi_{k}(\lam)}{2}\Big)].
\end{align*}
The estimation is almost the same as the real case.

The term
\begin{align*}
P(\frac{\partial}{\partial \lam})(2\pi)^{-(\frac{2s+i\vpi_{k}(\lam)}{2})}
\end{align*}
is dominated by
\begin{align*}
C_{1}(|\lam|+1)^{a}(2\pi)^{-(\frac{2s+i\vpi_{k}(\lam)}{2})}
\end{align*}
for some $a>0$ and constant $C_{1}>0$.

For the term
\begin{align*}
P(\frac{\partial}{\partial \lam})\Gam\Big(\frac{2s+i\vpi_{k}(\lam)}{2}\Big),
\end{align*}
using Theorem \ref{derivativeofgamma} for the estimation on the derivative of $\Gam(z)$, it is dominated by
\begin{align*}
C_{2}(|\lam|+1)^{b}\Gam\Big(\frac{2s+i\vpi_{k}(\lam)}{2}\Big)
\end{align*}
for some $b>0$ and some constant $C_{2}>0$. Here we use the fact that $\log (z)$ is dominated by $C(|z|+1)$ for some constant $C$ if $\Re(z)$ is bigger than $\textrm{max}\{\frac{\vpi_{k}(\mu)}{2} | \quad 1\leq k\leq n,  \mu \in C^{\veps \rho_{B}}\}$.

Hence we only need to show that the following term is bounded
\begin{align*}
\sup_{\lam\in \Fa^{*}_{\veps}}(|\lam|+1)^{t}\prod_{k=1}^{n}(2\pi)^{-(\frac{2s+i\vpi_{k}(\lam)}{2})}\Gam\Big(\frac{2s+i\vpi_{k}(\lam)}{2}\Big).
\end{align*}

When $\lam\in \Fa^{*}_{\veps} = \Fa^{*}+iC^{\veps \rho }$, the real part of $\frac{2s+i\vpi_{k}(\lam)}{2}$ is bounded and lies in a compact set, so the function $(2\pi)^{-(\frac{2s+i\vpi_{k}(\lam)}{2})}$ is always bounded.
Using Theorem \ref{estimateforgamma} for the estimation for $\Gam(x+iy)$ for $x\in \BR$ fixed, we have
\begin{align*}
&\sup_{\lam\in \Fa^{*}_{\veps}}(|\lam|+1)^{t}\prod_{k=1}^{n}(2\pi)^{-(\frac{2s+i\vpi_{k}(\lam)}{2})}\Gam\Big(\frac{2s+i\vpi_{k}(\lam)}{2}\Big)\leq \\
&\sup_{\lam\in \Fa^{*}_{\veps}}C (|\lam|+1)^{t}(\sqrt{2\pi})^{n}
\prod_{k=1}^{n}[ |\frac{2\Im(s)+\vpi_{k}(x)}{2}|^{\frac{2\Re(s) -\vpi_{k}(y)-1}{2}}
\\
&\cdot e^{\frac{\vpi_{k}(y)}{2} -\frac{2\Re(s)}{2}-\frac{|2\Im(s)+\vpi_{k}(x)|\pi}{4}}]
\end{align*}
for some constant $C>0$. Here we write $\lam = x+iy$ with $x\in \Fa^{*}$, $y\in C^{\veps \rho}$.

Now we know that $s\in \BC$ is fixed, and $y$ lies in $C^{\veps\rho}$, which is a compact set. The term $\vpi_{k}(\lam)$ is also dominated by a polynomial function in $|\lam|+1$. Therefore up to a constant and a polynomial in $(|\lam|+1)$, we only need to evaluate the following term
\begin{align*}
\sup_{x\in \Fa^{*}}(|x|+1)^{t}\prod_{k=1}^{n}e^{-\frac{|\vpi_{k}(x)|\pi}{4}}.
\end{align*}
By Lemma \ref{assumpweight}, it is bounded by
\begin{align*}
\sup_{x\in \Fa^{*}}(|x|+1)^{t}\prod_{k=1}^{m}e^{-\frac{C_{\rho}|x_{k}|\pi}{4}}.
\end{align*}
which is bounded by a constant. This proves the theorem.
\end{proof}

We can also prove an asymptotic for $\Phi_{\psi,\rho,s}^{K}$.
By definition, the spherical component of $\Phi_{\psi,\rho,s}^{K}$ is determined via the following identity
\begin{align*}
\CH(\Phi_{\psi,\rho,s}^{K}) = \frac{L(1+s+\frac{l}{2},\pi,\rho)}{L(-s-\frac{l}{2},\pi^{\vee},\rho)}.
\end{align*}
Here we notice that if $\pi$ has Langlands parameter
\begin{align*}
t \to   \left( \begin{matrix} 
      |t|^{i\lam_{1}} &  			   &	&\\
       			   & |t|^{i\lam_{2}} &	&\\
       			   &			   &	&\\
       			   &			   &\ddots	&\\
			   &			   &	&|t|^{i\lam_{m}}  \\
   \end{matrix} \right),
\end{align*}
then $\pi^{\vee}$ has Langlands parameter
\begin{align*}
t\to   \left( \begin{matrix} 
      |t|^{-i\lam_{1}} &  			   &	&\\
       			   & |t|^{-i\lam_{2}} &	&\\
       			   &			   &	&\\
       			   &			   &\ddots	&\\
			   &			   &	&|t|^{-i\lam_{m}}  \\
   \end{matrix} \right).
\end{align*}

We first simplify the expression for $\gam$-factor
\begin{lem}
The formula $\CH(\Phi_{\psi,\rho,s}^{K}) =
\frac{L(1+s+\frac{l}{2},\pi,\rho)}{L(-s-\frac{l}{2},\pi^{\vee},\rho)}$ can be simplified to be
\begin{align*}
\prod_{k=1}^{n} [\pi^{-(\frac{1}{2}+2s+\frac{l}{2}+i\vpi_{k}(\lam)    ) }
& \Gam\Big(\frac{1+2s+\frac{l}{2}+i\vpi_{k}(\lam)}{2} \Big)
 \\
 \frac{1}{\pi}\sin\Big(\pi(\frac{2+2s+\frac{l}{2}+i\vpi_{k}(\lam)}{2})\Big)
&\Gam\Big(\frac{2+2s+\frac{l}{2}+i\vpi_{k}(\lam)}{2}\Big)].
\end{align*}
\end{lem}
\begin{proof}
Using the definition of $L$-function, we have
\begin{align*}
\frac{L(1+s+\frac{l}{2},\pi,\rho)}{L(-s-\frac{l}{2},\pi^{\vee},\rho)}=
\frac{\prod_{k=1}^{n}\pi^{-\frac{1+2s+\frac{l}{2}+i\vpi_{k}(\lam)}{2}} \Gam(\frac{1+2s+\frac{l}{2}+i\vpi_{k}(\lam)}{2} ) }
{\prod_{k=1}^{n}  \pi^{\frac{2s+\frac{l}{2}+i\vpi_{k}(\lam)}{2}} \Gam(  -\frac{2s+\frac{l}{2}+i\vpi_{k}(\lam)}{2})       }
\\=
\prod_{k=1}^{n} \pi^{-(\frac{1}{2}+2s+\frac{l}{2}+i\vpi_{k}(\lam)    ) }
\frac{ \Gam(\frac{1+2s+\frac{l}{2}+i\vpi_{k}(\lam)}{2} )}{\Gam(  -\frac{2s+\frac{l}{2}+i\vpi_{k}(\lam)}{2}) }.
\end{align*}
Using the functional equation for $\Gam(z)$
\begin{align*}
\Gam(z)\Gam(1-z) = \frac{\pi}{\sin(\pi z)}
\end{align*}
we get
\begin{align*}
\frac{1}{\Gam\big(-\frac{2s+\frac{l}{2}+i\vpi_{k}(\lam)}{2}\big)} = \frac{1}{\pi}\sin\Big(\pi(\frac{2+2s+\frac{l}{2}+i\vpi_{k}(\lam)}{2})\Big)
\Gam\Big(\frac{2+2s+\frac{l}{2}+i\vpi_{k}(\lam)}{2}\Big).
\end{align*}
It follows that
\begin{align*}
\prod_{k=1}^{n} \pi^{-(\frac{1}{2}+2s+\frac{l}{2}+i\vpi_{k}(\lam)    ) }
&\frac{ \Gam\big(\frac{1+2s+\frac{l}{2}+i\vpi_{k}(\lam)}{2} \big)}{\Gam\big(  -\frac{2s+\frac{l}{2}+i\vpi_{k}(\lam)}{2}\big) }
\\=
\prod_{k=1}^{n} [\pi^{-(\frac{1}{2}+2s+\frac{l}{2}+i\vpi_{k}(\lam)    ) }
 &\Gam\Big(\frac{1+2s+\frac{l}{2}+i\vpi_{k}(\lam)}{2} \Big)
 \\
 \frac{1}{\pi}\sin\Big(\pi(\frac{2+2s+\frac{l}{2}+i\vpi_{k}(\lam)}{2})\Big)
&\Gam\Big(\frac{2+2s+\frac{l}{2}+i\vpi_{k}(\lam)}{2}\Big)].
\end{align*}
\end{proof}

We write $\lam= x+iy$ with $x\in \Fa^{*}$ and $y\in C^{\veps\rho_{B}}$, and we notice that the function $ \frac{1}{\pi}\sin\Big(\pi(\frac{2+2s+\frac{l}{2}+i\vpi_{k}(\lam)}{2})\Big)$ is a Paley-Wiener function in $\lam$, hence lies in $\CS(\Fa^{*}_{\veps})$. The function $\pi^{-(\frac{1}{2}+2s+\frac{l}{2}+i\vpi_{k}(\lam)    ) }$ is bounded.
Then combining with Theorem \ref{asympforbasiccplx} and the fact that $\CS(\Fa^{*}_{\veps})$ is a Fr\'echet algebra,
we know that if $\Re(2s+1+\frac{l}{2})$ is bigger than $\textrm{max}\{\vpi_{k}(\mu) | \quad 1\leq k\leq n,  \mu \in C^{\veps \rho_{B}}\}$ and $\Re(2s+2+\frac{l}{2})$ is bigger than $\textrm{max}\{\vpi_{k}(\mu) | 1\leq k\leq n,  \mu \in C^{\veps \rho_{B}}\}$, the function $\CH(\Phi_{\psi,\rho,s}^{K})$ lies in $\CS(\Fa^{*}_{\veps})$. Using the fact that $\pi_{\lam}\cong \pi_{w\lam}$ for $w\in W$, we know that $\CH(\Phi_{\psi,\rho,s}^{K})$ lies in $\CS(\Fa^{*}_{\veps})^{W}$.

In other words, we have proved the following asymptotic for $\Phi_{\psi,\rho,s}^{K}$
\begin{thm}
If $\Re(s)$ satisfies the following inequality
\begin{align*}
\Re(s)>-\frac{1}{2}-\frac{l}{4}+\textrm{max}\{\frac{\vpi_{k}(\mu)}{2} | \quad 1\leq k\leq n,  \mu \in C^{\veps \rho_{B}}\},
\end{align*}
then the function $\Phi_{\psi,\rho,s}^{K}$ lies in $\CS^{p}(K\bs G/K)$.
\end{thm}

\bibliographystyle{amsalpha}
\bibliography{zhilin}

\providecommand{\bysame}{\leavevmode\hbox to3em{\hrulefill}\thinspace}
\providecommand{\MR}{\relax\ifhmode\unskip\space\fi MR }
\providecommand{\MRhref}[2]{%
  \href{http://www.ams.org/mathscinet-getitem?mr=#1}{#2}
}
\providecommand{\href}[2]{#2}
\begin{thebibliography}{FLM11}

\bibitem[Ank91]{anker1991spherical}
Jean-Philippe Anker, \emph{The spherical {F}ourier transform of rapidly
  decreasing functions. {A} simple proof of a characterization due to
  {H}arish-{C}handra, {H}elgason, {T}rombi, and {V}aradarajan}, J. Funct. Anal.
  \textbf{96} (1991), no.~2, 331--349.

\bibitem[Bat53]{hightransfun}
Harry Bateman, \emph{Higher transcendental functions vol. 1}, Tata-mcgrawhill
  Book Company Ltd., Bombay, 1953.

\bibitem[BK00]{BK00}
A.~Braverman and D.~Kazhdan, \emph{{$\gamma$}-functions of representations and
  lifting}, Geom. Funct. Anal. (2000), no.~Special Volume, Part I, 237--278,
  With an appendix by V. Vologodsky, GAFA 2000 (Tel Aviv, 1999).

\bibitem[BK03]{braverman2003sheaves}
Alexander Braverman and David Kazhdan, \emph{{$\gamma$}-sheaves on reductive
  groups}, Studies in memory of {I}ssai {S}chur ({C}hevaleret/{R}ehovot, 2000),
  Progr. Math., vol. 210, Birkh\"auser Boston, Boston, MA, 2003, pp.~27--47.

\bibitem[BNS16]{bouthier2016formal}
A.~Bouthier, B.~C. Ng\^o, and Y.~Sakellaridis, \emph{On the formal arc space of
  a reductive monoid}, Amer. J. Math. \textbf{138} (2016), no.~1, 81--108.

\bibitem[BNS17]{formalarcerrtum}
\bysame, \emph{Erratum to: ``{O}n the formal arc space of a reductive monoid''
  [ {MR}3462881]}, Amer. J. Math. \textbf{139} (2017), no.~1, 293--295.

\bibitem[Bor98]{borelssgrp}
Armand Borel, \emph{Semisimple groups and {R}iemannian symmetric spaces}, Texts
  and Readings in Mathematics, vol.~16, Hindustan Book Agency, New Delhi, 1998.

\bibitem[Bum13]{bumplietheory}
Daniel Bump, \emph{Lie groups}, second ed., Graduate Texts in Mathematics, vol.
  225, Springer, New York, 2013.

\bibitem[Car79]{cartier79}
P.~Cartier, \emph{Representations of {$p$}-adic groups: a survey}, Automorphic
  forms, representations and {$L$}-functions ({P}roc. {S}ympos. {P}ure {M}ath.,
  {O}regon {S}tate {U}niv., {C}orvallis, {O}re., 1977), {P}art 1, Proc. Sympos.
  Pure Math., XXXIII, Amer. Math. Soc., Providence, R.I., 1979, pp.~111--155.

\bibitem[Che16]{chen2016non}
Tsao-Hsien Chen, \emph{Non-linear fourier transforms and the braverman-kazhdan
  conjecture}, arXiv preprint arXiv:1609.03221 (2016).

\bibitem[CMc82]{subrepresentation}
William Casselman and Dragan Mili\v~ci\'c, \emph{Asymptotic behavior of matrix
  coefficients of admissible representations}, Duke Math. J. \textbf{49}
  (1982), no.~4, 869--930.

\bibitem[CN17]{cheng2017conjecture}
Shuyang Cheng and Bảo~Ch{\^a}u Ng{\^o}, \emph{On a conjecture of braverman
  and kazhdan}, International Mathematics Research Notices (2017), rnx052.

\bibitem[FL11]{finis2011continuity}
Tobias Finis and Erez Lapid, \emph{On the continuity of {A}rthur's trace
  formula: the semisimple terms}, Compos. Math. \textbf{147} (2011), no.~3,
  784--802.

\bibitem[FL16]{finis2016continuity}
\bysame, \emph{On the continuity of the geometric side of the trace formula},
  Acta Math. Vietnam. \textbf{41} (2016), no.~3, 425--455.

\bibitem[FLM11]{finis2011spectral}
Tobias Finis, Erez Lapid, and Werner M\"uller, \emph{On the spectral side of
  {A}rthur's trace formula---absolute convergence}, Ann. of Math. (2)
  \textbf{174} (2011), no.~1, 173--195.

\bibitem[Get15]{getz2015nonabelian}
Jayce~R Getz, \emph{Nonabelian fourier transforms for spherical
  representations}, arXiv preprint arXiv:1506.09128 (2015).

\bibitem[GJ72]{gjzeta}
Roger Godement and Herv\'e Jacquet, \emph{Zeta functions of simple algebras},
  Lecture Notes in Mathematics, Vol. 260, Springer-Verlag, Berlin-New York,
  1972.

\bibitem[Gro98]{grosssatake}
Benedict~H. Gross, \emph{On the {S}atake isomorphism}, Galois representations
  in arithmetic algebraic geometry ({D}urham, 1996), London Math. Soc. Lecture
  Note Ser., vol. 254, Cambridge Univ. Press, Cambridge, 1998, pp.~223--237.

\bibitem[GV88]{gangolli2012harmonic}
Ramesh Gangolli and V.~S. Varadarajan, \emph{Harmonic analysis of spherical
  functions on real reductive groups}, Ergebnisse der Mathematik und ihrer
  Grenzgebiete [Results in Mathematics and Related Areas], vol. 101,
  Springer-Verlag, Berlin, 1988.

\bibitem[Jac79]{jacquet79}
Herv\'e Jacquet, \emph{Principal {$L$}-functions of the linear group},
  Automorphic forms, representations and {$L$}-functions ({P}roc. {S}ympos.
  {P}ure {M}ath., {O}regon {S}tate {U}niv., {C}orvallis, {O}re., 1977), {P}art
  2, Proc. Sympos. Pure Math., XXXIII, Amer. Math. Soc., Providence, R.I.,
  1979, pp.~63--86.

\bibitem[Kat82]{inversekato}
Shin-ichi Kato, \emph{Spherical functions and a {$q$}-analogue of {K}ostant's
  weight multiplicity formula}, Invent. Math. \textbf{66} (1982), no.~3,
  461--468.

\bibitem[Kna94]{knapp55local}
A.~W. Knapp, \emph{Local {L}anglands correspondence: the {A}rchimedean case},
  Motives ({S}eattle, {WA}, 1991), Proc. Sympos. Pure Math., vol.~55, Amer.
  Math. Soc., Providence, RI, 1994, pp.~393--410.

\bibitem[Laf14]{lafforgue2014noyaux}
Laurent Lafforgue, \emph{Noyaux du transfert automorphe de {L}anglands et
  formules de {P}oisson non lin\'eaires}, Jpn. J. Math. \textbf{9} (2014),
  no.~1, 1--68.

\bibitem[Lan70]{langlandsproblems}
R.~P. Langlands, \emph{Problems in the theory of automorphic forms}, 18--61.
  Lecture Notes in Math., Vol. 170.

\bibitem[Lan89]{langlands1989irreducible}
\bysame, \emph{On the classification of irreducible representations of real
  algebraic groups}, Representation theory and harmonic analysis on semisimple
  {L}ie groups, Math. Surveys Monogr., vol.~31, Amer. Math. Soc., Providence,
  RI, 1989, pp.~101--170.

\bibitem[Lan04]{langlands2004beyond}
Robert~P. Langlands, \emph{Beyond endoscopy}, Contributions to automorphic
  forms, geometry, and number theory, Johns Hopkins Univ. Press, Baltimore, MD,
  2004, pp.~611--697.

\bibitem[Li17]{wwlibasic}
Wen-Wei Li, \emph{Basic functions and unramified local {$L$}-factors for split
  groups}, Sci. China Math. \textbf{60} (2017), no.~5, 777--812.

\bibitem[Lus83]{inverselusztig}
George Lusztig, \emph{Singularities, character formulas, and a {$q$}-analog of
  weight multiplicities}, Analysis and topology on singular spaces, {II}, {III}
  ({L}uminy, 1981), Ast\'erisque, vol. 101, Soc. Math. France, Paris, 1983,
  pp.~208--229.

\bibitem[Ng{\^o}16]{ngo2016hankel}
Bao~Ch{\^a}u Ng{\^o}, \emph{Hankel transform, langlands functoriality and
  functional equation of automorphic l-functions}, Takagi lectures \textbf{18}
  (2016), 1--19.

\bibitem[Sak14]{yiannisinverse}
Yiannis Sakellaridis, \emph{Inverse satake transforms}, arXiv preprint
  arXiv:1410.2312 (2014).

\bibitem[Sat63]{satake63}
Ichir\^o Satake, \emph{Theory of spherical functions on reductive algebraic
  groups over {${\Fp}$}-adic fields}, Inst. Hautes \'Etudes Sci. Publ. Math.
  (1963), no.~18, 5--69.

\bibitem[Tat50]{tatethesis}
John~Torrence Tate, Jr, \emph{F{OURIER} {ANALYSIS} {IN} {NUMBER} {FIELDS} {AND}
  {HECKE}'{S} {ZETA}-{FUNCTIONS}}, ProQuest LLC, Ann Arbor, MI, 1950, Thesis
  (Ph.D.)--Princeton University.

\bibitem[Vin95]{vinberg1995reductive}
E.~B. Vinberg, \emph{On reductive algebraic semigroups}, Lie groups and {L}ie
  algebras: {E}. {B}. {D}ynkin's {S}eminar, Amer. Math. Soc. Transl. Ser. 2,
  vol. 169, Amer. Math. Soc., Providence, RI, 1995, pp.~145--182.

\end{thebibliography}

\end{document}